\documentclass[12pt,oneside,reqno]{amsart}
\usepackage[utf8]{inputenc}
\usepackage[english]{babel}

\usepackage[margin=1in]{geometry}

\usepackage{multicol}

\usepackage{amsmath}
\usepackage{amsthm}
\usepackage{amssymb}
\usepackage{graphicx}
\usepackage{mathrsfs}
\usepackage[colorinlistoftodos]{todonotes}
\usepackage[colorlinks=true, allcolors=black]{hyperref}
\usepackage{comment}
\usetikzlibrary{graphs}
\usepackage{float}
\usepackage{xcolor}
\usepackage{tikz-cd}
\usepackage[normalem]{ulem}
\usepackage{manfnt}
\newcommand{\showcomments}{yes}

\newsavebox{\commentbox}
\newenvironment{com}%
{\ifthenelse{\equal{\showcomments}{yes}}%
{\footnotemark
        \begin{lrbox}{\commentbox}
        \begin{minipage}[t]{1in}\raggedright\sffamily\tiny
        \footnotemark[\arabic{footnote}]}
{\begin{lrbox}{\commentbox}}}%
{\ifthenelse{\equal{\showcomments}{yes}}%
{\end{minipage}\end{lrbox}\marginpar{\usebox{\commentbox}}}
{\end{lrbox}}}

\newtheorem{thm}{Theorem}[section]

\newtheorem{theorem}[thm]{Theorem}
\newtheorem{corollary}[thm]{Corollary}
\newtheorem{lemma}[thm]{Lemma}
\newtheorem{proposition}[thm]{Proposition}

\newtheorem{claim}[thm]{Claim}
\newtheorem*{theorem*}{Theorem}

\theoremstyle{definition}
\newtheorem{construction}[thm]{Construction} 

\newtheorem{definition}[thm]{Definition}

\newtheorem{prob}[thm]{Problem}

\theoremstyle{remark}

\newtheorem{remark}[thm]{Remark}
\newtheorem{example}[thm]{Example}
\newtheorem{examples}[thm]{Examples}

\newcommand{\nclose}[1]{\ensuremath{\langle\!\langle#1\rangle\!\rangle}}
\newcommand{\dist}{\textup{\textsf{d}}}
\newcommand{\scname}[1]{\text{\sf #1}}
\newcommand{\area}{\scname{Area}}

\newcommand{\field}[1]{\mathbb{#1}}
\newcommand{\integers}{\ensuremath{\field{Z}}}
\newcommand{\rationals}{\ensuremath{\field{Q}}}
\newcommand{\naturals}{\ensuremath{\field{N}}}
\newcommand{\reals}{\ensuremath{\field{R}}}

\DeclareMathOperator{\link}{link}

\title[Asphericity of cubical presentations: the general case]{Asphericity of cubical presentations:   the general case}

\author{Macarena Arenas}
\address{DPMMS, Centre for Mathematical Sciences, Wilberforce Road, Cambridge, CB3 0WB, UK
 and Clare College, University of Cambridge, Cambridge, CB2 1TL, UK}
\email{mcr59@dpmms.cam.ac.uk}

\thanks{The author was supported by the Denman Baynes Research Fellowship at Clare College, Cambridge.}

\begin{document}

\begin{abstract}
We show that, under suitable hypotheses, the coned-off spaces associated to $C(9)$ cubical small-cancellation presentations are aspherical, and use this to provide classifying spaces, or classifying spaces for proper actions, for their fundamental groups. 
 Along the way, 
 we show that the Cohen-Lyndon property holds for the subgroups of the fundamental group of a non-positively curved cube complex associated to  a $C(9)$ cubical presentation, and thus obtain near-sharp upper and lower bounds for the (rational) cohomological dimension of these quotients.

We apply these results to give an alternative construction of compact $K(\pi,1)$ for Artin  groups with no labels in $\{3,4\}$, from which a new direct sum decomposition for their homology and cohomology with various coefficients above a certain dimension follows. We also address a question of Wise about the virtual torsion-freeness of cubical small-cancellation groups.
\end{abstract}

\maketitle
\tableofcontents

\newpage

\section{Introduction}

The purpose of this work is to prove a number of asphericity results for the quotients of cubulated groups that arise from cubical presentations $\langle X \mid \{Y_i \longrightarrow X\} \rangle$ where $X$ is a compact, non-positively curved cube complex and $Y_i \longrightarrow X$ are local isometries of connected non-positively curved cube complexes. A cubical presentation is a natural generalisation of a group presentation, and just like group presentations can be studied via their presentation complexes, a cubical presentation has a canonically associated \emph{coned-off space} $X^*$ whose fundamental group $\pi_1X^*$ is the quotient $\Gamma=\pi_1X/\nclose{\{\pi_1Y_i\}}$. We prove that if $\langle X \mid \{Y_i\} \rangle$ satisfies the cubical $C(9)$ small-cancellation condition, then a certain quotient of $\widetilde X^*$, called the \emph{reduced space} $\bar X^*$, is  aspherical. Using the asphericity of $\bar X^*$, we provide sufficient conditions under which  a compact classifying space,  or a cocompact classifying space for proper actions for $\Gamma$, can be obtained from $X^*$.

Concretely, we show:

\begin{theorem}\label{thm:mainintro}[Theorem~\ref{thm:asph}]
Let $X^*=\langle X \mid \{Y_1, \ldots, Y_n\} \rangle$ be a  cubical presentation. If $X^*$ satisfies the cubical $C(9)$ condition, then $\bar X^*$ is aspherical. In particular, 
 if $X^*$ is reduced, then $X^*$ is a classifying space for $\pi_1 X^*$.
\end{theorem}

See Subsection~\ref{subsec:cubsc} and Section~\ref{sec:mains} for the definitions. From Theorem~\ref{thm:mainintro}, we deduce a bordification result, Corollary~\ref{cor:asphericalreloaded}, that yields  sufficient conditions under which $\pi_1X/\nclose{\{\pi_1Y_i\}}$ has finiteness  type $F_n$ or $F_\infty$, thus providing a topological conterpart to Petrosyan and Sun's~\cite[2.2 (iii)]{PetSun21}  in this context.
 
\begin{corollary}\label{cor:finpropintro}[Corollary~\ref{cor:finprop}]
Let $\Gamma=\pi_1X^*$ where $X^*=\langle X \mid \{Y_i\}_{i \in I} \rangle$  is a  cubical presentation that satisfies the $C(9)$ condition and each $Y_i$ is compact. If $\pi_1X$  and each $Stab_{\pi_1X}(\widetilde Y_i)/\pi_1Y_i$ is  type $F_n$ (resp., type $F$), then $\Gamma$ is type $F_n$ (resp., type $F$).
\end{corollary}

 Just like in the classical small-cancellation case, the group defined by a cubical small-cancellation presentation might have torsion. When this torsion arises in a controlled way, we can also describe classifying spaces for proper actions for $\Gamma$:

\begin{theorem}\label{thm:mainintro2}[Theorem~\ref{thm:asphfinite}]
 Let $X^*=\langle X \mid \{Y_1, \ldots, Y_n\} \rangle$ be a  cubical presentation that satisfies the  $C(9)$ condition.  If $[Stab_{\pi_1X}(\widetilde Y_i):\pi_1Y_i]< \infty$ for each $i \in I$, then the reduced space $\bar X^*$ is a model for $\underbar E\pi_1 X^*$, the classifying space for proper actions for $\pi_1 X^*$. 
\end{theorem}

The motivation for finding good models for $\underbar E\Gamma$ is twofold: on the one hand, the dimension of an $\underbar E\Gamma$ is an upperbound for the rational cohomological dimension of $\Gamma$; on the other hand, classifying space for proper actions arise in the statement of the Baum-Connes conjecture.%, which posits a connection between  the $K$-theory of the reduced $C^*$-algebra of a group and the $K$-homology of the classifying space for proper actions of that group.

Using Theorem~\ref{thm:mainintro2}, we address a question posed by Wise in~\cite[4.5]{WiseIsraelHierarchy} about the virtual torsion-freeness of cubical small-cancellation quotients:

\begin{lemma}\label{lem:vcdintro}[Lemma~\ref{lem:vcd}] Let $X^*=\langle X \mid \{Y_i\}_{i \in I} \rangle$ be a  $C(9)$ cubical presentation. If $\pi_1X$ is virtually special and $[Stab_{\pi_1X}(\widetilde Y_i):\pi_1Y_i]< \infty$ for each  $Y_i \rightarrow X$, then $\pi_1X^*$ is virtually torsion-free.
\end{lemma}

Asphericity results are available in various small-cancellation theories: it was shown by Lyndon~\cite{Lyn66} that classical (torsion-free) $C(6)$ presentations are aspherical, and thus that the groups admitting such presentations have cohomological dimension at most equal to $2$. This is also the case for graphical small-cancellation quotients~\cite{Gromov2003,Ollivier06};    in the higher-dimensional setting,  Coulon exhibited classifying spaces for rotating families of groups in~\cite{Cou11}. Classical and graphical small-cancellation groups are hyperbolic as soon as they satisfy either the  $C'(\frac{1}{6})$ or the $C(7)$ conditions, and the groups produced in \cite{Cou11} are also hyperbolic. Thus, our main theorem goes significantly beyond the currently available methods, in that it serves to analyse quotients of  arbitrarily large cohomological dimension, and which might not be even relatively hyperbolic.

 Theorems~\ref{thm:mainintro} and~\ref{thm:mainintro2} are proven by establishing a Cohen-Lyndon property for $C(9)$ cubical presentations:

\begin{theorem}\label{thm:clintro}[Theorem~\ref{thm:clprop}]
Let $X^*=\langle X \mid \{Y_i\}_{i \in I} \rangle$ be a cubical presentation. If $X^*$ satisfies the $C(9)$ condition, then the pair $(\pi_1X, \{\pi_1 Y_i\}_{i \in I})$ satisfies the Cohen-Lyndon property: there exist full left transversals $T_i$ of  $\mathcal{N}( \pi_1 Y_i )\nclose{\{\pi_1 Y_i\}}$  such that $$\nclose{\{\pi_1 Y_i\}}= \ast_{i \in I, t \in T_i} \langle \pi_1 Y_i\rangle ^t$$
where $\mathcal{N}( \pi_1 Y_i )$ denotes the normaliser of $\pi_1 Y_i$ in $\pi_1X$.
\end{theorem}

Theorem~\ref{thm:clintro} is of independent interest.   It follows a long list of results proving versions of the Cohen-Lyndon property in various settings~\cite{KS72, MCS86, EH87, DH94, GHMOSW24}, the most general of which is probably Sun's~\cite{Sun20} Cohen-Lyndon property for triples, which we discuss in detail in Subsection~\ref{subsec: triples}. We note, however, that none of the available versions of the Cohen-Lyndon property imply our result.

 It has been observed many times in the literature (see the references in the previous paragraph, and also~\cite{LS77}) that suitable versions of the Cohen-Lyndon property yield a wealth of algebraic information. Notably, for a quotient $G / \nclose{\mathcal{H}}$ satisfying a Cohen-Lyndon-type theorem, one can  deduce \begin{itemize}

 \item Structure theorems for the (relative) relation module $Rel(G,\nclose{\mathcal{H}})$ (in analogy with the Lyndon Identity Theorem in 1-relator groups).

\item A characterisation of the torsion in $G / \nclose{\mathcal{H}}$ in terms of that of the $\mathcal{N}(H_i)/H_i$'s.
 \item Expressions for the (co)homology of $G / \nclose{\mathcal{H}}$ in terms of that of $G$ and the $\mathcal{N}(H_i)/H_i$'s (above a certain dimension).
 \end{itemize}  

In particular, using Theorem~\ref{thm:clintro} and results from~\cite{PetSun21}, we immediately obtain:

\begin{corollary}\label{cor:CLcorintro}
Let $X^*=\langle X \mid \{Y_i\}_{i \in I} \rangle$ be a cubical presentation. If $X^*$ satisfies the $C(9)$ condition, then $ cd(\pi_1X^*) \geq cd(\pi_1X)-\max_i\{cd(\pi_1Y_i)\}$ and
\begin{enumerate}
\item if $X^*$ is reduced, then  $$ cd(\pi_1X^*)\leq \max\{cd(\pi_1X),\max_i\{cd(\pi_1Y_i)\}+1\},$$ 
\item if  $[Stab_{\pi_1X}(\widetilde Y_i):\pi_1Y_i]< \infty$ for each $i \in I$, then 
$$cd_\rationals(\pi_1X^*)\leq \max\{dim(X),\max_i\{dim(Y_i)\}+1\},$$
\item if  $\pi_1X$ is virtually special and $[Stab_{\pi_1X}(\widetilde Y_i):\pi_1Y_i]< \infty$ for each $i \in I$, then 
$$ vcd(\pi_1X^*)\leq \max\{cd(\pi_1X),\max_i\{cd(\pi_1Y_i)\}+1\}.$$
\end{enumerate} 

Moreover, for every $\pi_1 X^*$-module $A$
\begin{enumerate}
\item if $X^*$ is reduced, then for all $n\geq \max\{cd(\pi_1Y_i)\}+2$: $$H^n(\pi_1X^*, A)=H^n(\pi_1X, A),$$ 
\item otherwise, for all $n\geq \max\{\mathcal{N}(\pi_1Y_i)/cd(\pi_1Y_i)\}+2$: $$H^n(\pi_1X^*, A)=H^n(\pi_1X, A)\oplus \bigoplus_i H^n(\mathcal{N}(\pi_1Y_i)/\pi_1Y_i, A).$$
\end{enumerate}
And analogously for $H_n(\pi_1X^*, A)$.
\end{corollary}

\subsection{Applications} The results in this work can be applied  to a host of examples coming from the hyperbolic and relatively hyperbolic settings. In particular, Theorem~\ref{thm:mainintro} provides an alternative proof of the cohomological dimension bounds obtained by the author in~\cite{Arenas2023} for certain examples of exotic hyperbolic groups. Theorems~\ref{thm:mainintro},~\ref{thm:mainintro2}, and~\ref{thm:clintro}, together with their corollaries, can also be applied to the cubical presentations obtained in~\cite{WiseIsraelHierarchy} for high-power cyclic quotients of cubulated hyperbolic groups and in~\cite{FuterWise16} for random (in a suitable sense) quotients of hyperbolic cubulated groups.

Our results also have applications in the absence of any strong form of hyperbolicity: 
using CAT(0) geometry, Charney proved in~\cite{Charney00} that the $K(\pi,1)$ conjecture -- which remains open in general -- holds for all locally reducible Artin groups. This class of Artin groups contains in particular all Artin groups with  no labels in $\{3,4\}$. In Subsection~\ref{subsecartin}, we give cubical small-cancellation presentations to which Theorem~\ref{thm:mainintro} can be applied for all Artin groups with no $\{3,4\}$ labels. This provides a new, purely topological description of finite models for their classifying spaces.

 Assuming the $K(\pi,1)$ conjecture, Davis and Leary computed the $\ell^2$-cohomology of Artin groups in~\cite{DavisLeary03}. For right-angled Artin groups, their cohomology with integer coefficients  can be easily computed from their respective Salvetti complexes; outside of the right-angled case, even if one assumes the $K(\pi,1)$ 
conjecture, there isn't even a calculation of the rational cohomology or homology of an arbitrary Artin group. Via Theorem~\ref{thm:clintro}, we obtain in Corollary~\ref{cor:artins}  explicit  formulas for many of the homology and cohomology groups with arbitrary coefficients of Artin groups with no $\{3,4\}$ labels.

We also describe cubical small-cancellation presentations for some families of Dyer groups, Shephard groups, and  other quotients of right-angled Artin groups,  and use these to derive similar corollaries as in the Artin group case. To the best of our knowledge, these results  constitute the first explicit cohomological  calculations for such classes. Dyer and Shephard groups have recently received a considerable amount of interest (see~\cite{Goldman23, Soergel24}, and the references therein).

\subsection{Theorem~\ref{thm:clintro}  and the Cohen-Lyndon property for triples}\label{subsec: triples}

In the paper~\cite{Sun20}, Sun obtains the following version of the Cohen-Lyndon property:

\begin{theorem}\label{thm:sun intro}
Let  $\Gamma$ be a group and $\mathcal{H}=\{H_1, \ldots, H_n\}$ be a family of hyperbolically embedded subgroups of $\Gamma$. Then the triple $(\Gamma, \mathcal{H}, \mathcal{K})$  where $\mathcal{K}=\{K_1 \triangleleft H_1, \ldots, K_n \triangleleft H_n\}$ has
the Cohen-Lyndon property  for all sufficiently deep $K_i \subset H_i$ in the following sense: there exist full left transversals $T_i$ of  $H_i\nclose{\mathcal{K}}$  such that $$\nclose{\mathcal{K}}= \ast_{i \in I, t \in T_i} \langle K_i\rangle ^t.$$
\end{theorem}

We will not define hyperbolically embedded subgroups here, but note that the hypotheses of the theorem are satisfied if $\Gamma$ is hyperbolic relative to $\mathcal{H}$. 

The main difference between Theorem~\ref{thm:clintro} and Sun's result is that in the former, we are able to prescribe the collection $\mathcal{K}$, and in the later, one can prescribe the $\mathcal{H}$, but the $K_i$ depend on the $H_i$. 
Thus, Theorem~\ref{thm:sun intro} cannot be used to deduce asphericity for $X^*$. Indeed, if  $\mathcal{K}$ is a fixed collection of subgroups of $\Gamma$, and we wish to understand the quotient $\Gamma/\nclose{\mathcal{K}}$, then having  control over a quotient $\Gamma/\nclose{\mathcal{K}'}$ where the  $K_i' \in  \mathcal{K}'$ are subgroups of the $K_i \in \mathcal{K}$  gives us no information about the structure of the original quotient. This is already evident for  a ``classical" group presentation $\mathcal{P}=\langle s_1, \ldots\, s_k \mid r_i, \ldots, r_\ell \rangle$,  where one can modify the $r_i$'s  by raising them to large enough proper powers to obtain a small-cancellation presentation $\mathcal{P'}$, but the group defined by $\mathcal{P'}$, generically, has no relation whatsoever with the group defined by $\mathcal{P}$. Sun's Theorem is a powerful tool for producing examples, but does not allow us to work directly with a prescribed quotient $\Gamma/\langle\langle H \rangle \rangle$.

In the cubical setting, an instance in which it is clear that direct control over $\pi_1 X^*=\pi_1X/\nclose{\pi_1Y_i}$ is essential -- rather than control over some deeper quotient of $\pi_1 X$ --  is when considering examples of well-known classes of groups that arise naturally as  fundamental groups of cubical presentations, and that are ``rigid" in the sense that almost any modification to the subgroups being quotiented yields a group that is not within the desired class. In this paper, the main examples of this are given by Artin groups and their generalisations (Dyer and Shephard groups). See Subsection~\ref{subsecartin} for these constructions.

\subsection{Strategy and structure}
In Section~\ref{sec:back} we give the necessary background on cubical presentations and diagrams. We also define the Cohen-Lyndon property and briefly survey some instances in which it is known to hold.
In Section~\ref{sec:dis} we define a certain topological cover (the ``three-way decomposition'') of the cubical part $\hat X\subset \widetilde X^*$, which should be thought of as a cubical Cayley graph, and analyse the connectivity properties of intersection of elements in this cover. 
In Section~\ref{sec:order} we use these results to ``rebuild'' the space $\hat X$ in such a way that we retain control over its homotopy type after each step in the reconstruction.   
In Section~\ref{sec:mains} we prove Theorem~\ref{thm:clintro}. The proof  is topological, and relies heavily on the  tools of cubical small-cancellation theory and on the strategy of the proof of the main theorem in~\cite{arenas2023cohenlyndon}.
In Section~\ref{sec:mains} we explain how to derive asphericity results from Theorem~\ref{thm:clintro}. In Section~\ref{sec:applications}, we outline some applications, and in Section~\ref{sec:comments} we state some questions and problems.

\subsection*{Acknowledgements} I would like to thank Henry Wilton, Daniel Groves, and Mark Hagen for many suggestions and conversations, Ruth Charney, Ian Leary, and Stefan Witzel  for pointing out an inaccuracy in a preliminary version of the description of the classifying spaces in Subsection~\ref{subsecartin}, Kevin Schreve for pointing out a mistake in the statement of Corollary~\ref{cor:artins}, and Yago Antol\'in, Rachael Boyd, Martin Blufstein, Ruth Charney, Ian Leary, Nicolas Vaskou, and especially Alexandre Martin for explanations about various aspects of Artin groups.

\section{Background}\label{sec:back}

We assume familiarity with the basic background on cube complexes. Particularly with the definitions of \emph{non-positively curved cube complex}, \emph{local isometry}, \emph{midcube}, \emph{hyperplane}, and \emph{hyperplane carrier}. See for instance~\cite{Sageev95, GGTbook14, Arenas2023thesis}.

\subsection{Cubical small-cancellation notions} 
\label{subsec:cubsc}
Unless noted otherwise, all definitions and results concerning cubical small-cancellation theory recounted in this section originate in~\cite{WiseIsraelHierarchy}.

A \textit{cubical presentation} $\langle X \mid \{Y_i\}  \rangle$ consists of a connected non-positively curved cube complex $X$ together with a collection of local isometries of connected non-positively curved cube complexes $Y_i \overset{\varphi_i} \longrightarrow X$. Local isometries of non-positively curved cube complexes are $\pi_1$-injective, so it makes sense to define the \emph{fundamental group of a cubical presentation} as $\pi_1 X/\nclose{\{\pi_1 Y_i\}}$. This group is isomorphic to the fundamental group of the \emph{coned-off space} $X^*$ obtained  from the mapping cylinder $$(X \cup \{Y_i \times [0,1]\})/\{(y_i, 1)\sim\varphi_i(y_i)\}$$ by collapsing each $Y_i\times \{0\}$ to a point. 

We identify $X^*$  with the corresponding cubical presentation $\langle X \mid\{Y_i\}  \rangle$  and write $X^*= \langle X \mid\{Y_i\}  \rangle$ throughout.

\begin{examples} Readily available examples of cubical presentations include:
\begin{enumerate}

\item A group presentation $\langle a_1,  \ldots, a_s \mid r_1, \ldots, r_m \rangle$  can be interpreted cubically  by letting $X$ be a bouquet of \textit{s} circles and letting each $Y_i$ map to the path determined by $r_i$.
\item A graphical presentation,  in the sense of~\cite{RipsSegev87} and~\cite{Gromov2003}, where $X$ is a bouquet of circles, the $Y_i$ are finite graphs and the maps $Y_i \rightarrow X$ are graph immersions.
\item   For every non-positively curved cube complex $X$ there is a ``free" cubical presentation $X^*=\langle X \mid \ \rangle$ with fundamental group $\pi_1X=\pi_1X^*$.
\item For every non-positively curved cube complex $X$, if $\hat X \rightarrow X$ is a covering, then $X^*=\langle X \mid \hat X \rightarrow X \ \rangle$ is a cubical presentation. In particular, if $\hat X \rightarrow X$ is a finite-degree covering, then $\pi_1 X^*$ is finite.
\end{enumerate}
\end{examples}

We will discuss more sophisticated cubical presentations in Section~\ref{sec:applications}.

\begin{definition}[Elevations] Let $Y \rightarrow X$ be a map and $\hat X \rightarrow X$ a covering map. An \emph{elevation} $\hat Y \rightarrow \hat X$ is a map satisfying
\begin{enumerate}
\item $\hat Y$ is connected,
\item the composition $\hat{Y} \rightarrow Y \rightarrow X$ equals $\hat{Y} \rightarrow \hat X \rightarrow X$, and
\item assuming all maps involved are basepoint preserving, $\pi_1 \hat{Y}$ equals the preimage of $\pi_1 \hat{X}$ in $\pi_1 Y$.
\end{enumerate}
\end{definition}

 In the context that concerns us, elevations will always be either:  
\begin{enumerate}
\item Elevations of a map $Y \rightarrow X$ to the universal cover $\widetilde X \rightarrow X$, which are denoted $\widetilde Y \rightarrow X$, since $\widetilde Y$ is indeed a copy of the universal cover of $Y$ in $\widetilde X$. At times, we will distinguish various elevations using the action of $\pi_1X$ on $\widetilde X$. That is,  we choose a base elevation $\widetilde Y$ and  tag a translate $g\widetilde Y$ by the corresponding element $g \in \pi_1X$.
\item Elevations of $Y \rightarrow X$ to a covering space $\hat X \rightarrow X$ that arises as a subspace of $\widetilde{X^*}$, which are denoted $ Y \rightarrow X$. The choice of notation is motivated by Theorem~\ref{thm:embeds}, which implies  that under appropriate small-cancellation conditions (which are a standing assumption for the entirety of this text) elevations $Y \rightarrow X$ to $\hat X \rightarrow X$ are embeddings. As in the case of elevations to $\widetilde X$, we might distinguish various elevations to $\hat X$ using the action of $\pi_1X^*$ on $\widetilde{X^*}$. 
\end{enumerate}

In this paper, a path $\sigma \rightarrow X$ is always assumed to be a combinatorial path mapping to the 1-skeleton of $X$.

\begin{definition}[Pieces]\label{def:pieces}
Let $\langle X \mid \{Y_i\} \rangle$ be a cubical presentation.
An \emph{abstract contiguous cone-piece} of $Y_j$ in $Y_i$ is an intersection $\widetilde{Y}_j \cap \widetilde{Y}_i$ where either $i \neq j$ or where $i = j$
but $\widetilde{Y}_j \neq \widetilde{Y}_i$. A \emph{cone-piece} of $Y_j$ in $Y_i$ is a path $p \rightarrow P$ in an abstract contiguous cone-piece of $Y_j$ in $Y_i$.
An \emph{abstract contiguous wall-piece} of $Y_i$ is an intersection $N(H) \cap \widetilde{Y}_i$ where $N(H)$ is
the carrier of a hyperplane $H$ that is disjoint from $\widetilde{Y}_i$. To avoid having to deal with empty pieces, we shall assume that $H$ is dual to an edge with an endpoint on $\widetilde{Y}_i$. A \emph{wall-piece} of $Y_i$ is a path $p \rightarrow P$ in an abstract contiguous wall-piece of $Y_i$.

A \emph{piece} is either a cone-piece or a wall-piece.
\end{definition}

 In Definition~\ref{def:pieces}, two lifts of a cone $Y$ are considered identical if they differ by an element of $Stab_{\pi_1X}(\widetilde Y)$. This is analogous to the conventions of classical small cancellation theory, where overlaps between a relator and any of its cyclic permutations are not regarded as pieces.

\begin{definition}
Let $Y \rightarrow X$ be a local isometry. $Aut_X(Y)$ is the group of combinatorial automorphisms $\psi: Y \rightarrow Y$ such that the diagram below is commutative:
\[\begin{tikzcd}
Y \arrow[r, "\psi"] \arrow[rd] & Y \arrow[d] \\
                               & X          
\end{tikzcd}\]
If $Y$ is simply connected, then $Aut_X(Y)$ is equal to $Stab_{\pi_1X}(Y)$. In general, $Aut_X(Y)\cong (\mathcal{N}_{Aut_X(\widetilde Y)}\pi_1Y)/\pi_1Y$, where $\mathcal{N}_G(H)$ is the normaliser of $H$ in $G$. 

See~\cite{ArHag2021} for a detailed discussion.
\end{definition}

We assume in all that follows that $\pi_1 Y$ is non-trivial, and additionally, as is standard to assume in this framework~\cite[3.3]{WiseIsraelHierarchy}, that $\pi_1Y$ is normal in $Stab_{\pi_1X}(\widetilde Y)$. 
This is equivalent to requiring that each element of $Stab_{\pi_1X}(\widetilde Y)$ projects to an element of $Aut_X(Y)$. We note that the small-cancellation conditions assumed throughout this paper, together with the assumption that $\pi_1 Y$ is non-trivial, imply that $\pi_1Y$ is normal in $Stab_{\pi_1X}(\widetilde Y)$. Thus, for the results proven herein, there is no loss of generality in making this assumption. 

The $C(p)$  condition is now defined as in the classical case. Namely:

\begin{definition}
A cubical presentation $X^*$ satisfies the \textit{$C(p)$ small cancellation condition} if no essential closed path $\sigma \rightarrow X^*$ is the concatenation of fewer than $p$ pieces.
\end{definition}

\subsection{Diagrams in 2-complexes and cubical presentations}\label{subsec: diagrams}

We will concern ourselves with three (topological) types of diagrams: \begin{enumerate}
\item A \emph{disc diagram} $D$ is a compact contractible combinatorial 2-complex, together with an embedding $D \hookrightarrow S^2$. The \emph{boundary path} $\partial D$ is the attaching map of the 2-cell ``at infinity'', i.e., the unique 2-cell in $S^2-Im(Int(D))$.
\item An \emph{annular diagram} $A$ is a compact combinatorial 2-complex homotopy equivalent to $S^1$, together with an embedding $A \hookrightarrow S^2$, which induces a cellular structure on $S^2$. The \emph{boundary paths} $\partial_{in}A$ and $\partial_{out}A$ of $A$ are the attaching maps of the two 2-cells in this cellulation of $S^2$ that do not correspond to cells of $A$. 
\item A \emph{Möbius diagram} $M$ is a compact combinatorial 2-complex homotopy equivalent to $S^1$, together with an embedding $M \hookrightarrow \reals P^2$. As for disc and annular diagrams, the embeddings induces a cellular structure on $\reals P^2$, and the \emph{boundary path} $\partial M$ is the attaching map of the 2-cell  that does not correspond to a cell of $M$. 
\end{enumerate} 

 A disc diagram $D$ is \emph{collared} by an annular diagram $A $ if $\partial D=\partial_{in} A$; similarly, $D$ is collared by a Möbius diagram $M$ if $\partial D=\partial M$. We assume moreover that the disc diagram $D$  does not contain any $2$-cell of $A$ or $M$, respectively. The ``moreover'' is a slight deviation from the  notion of a collared digram considered elsewhere in the literature, but will be more convenient for our applications.

A \emph{disc (resp. annular, resp. Möbius) diagram in a complex X} is a combinatorial map $D \rightarrow X$  (resp. $A \rightarrow X$, resp. $M \rightarrow X$).
Finally, a \emph{square disc (resp. annular, resp. Möbius) diagram} is a disc (resp. annular, resp. Möbius) diagram that is also a cube complex.

\begin{definition}[Diagrammatical objects in $X^*$]
The coned-off space $X^*$ consists of $X$ with a cone on $Y_i$ attached to $X$ for each $i$. The vertices of the cones on $Y_i$'s are the \emph{cone-vertices} of $X^*$. The cellular structure of $X^*$ consists of all the original cubes of $X$, and the pyramids over cubes in $Y_i$ with a cone-vertex for the apex. Let $D \rightarrow X^*$ be a disc diagram in a cubical presentation. The vertices in $D$ which are mapped to the cone-vertices of $X^*$ are the \emph{cone-vertices} of $D$. 
Triangles in $D$ are naturally grouped into cyclic families meeting around a cone-vertex.
Each such family forms a subspace of $D$ that is a cone on its bounding cycle. A  \emph{cone-cell} of $D$ is a cone that arises in this way.
When analysing diagrams in a cubical presentation we ``forget'' the subdivided cell-structure of a cone-cell $C$ and regard it simply as a single $2$-cell.
\end{definition}

All disc diagrams in this paper are either square disc diagrams,  or  disc diagrams in the coned-off space associated to a cubical presentation. 

\begin{definition}[Square-disc behaviours]\label{def:squarebee}
 A \emph{dual curve} in a square disc diagram is a path that is a concatenation of midcubes. The 1-cells crossed by a dual curve are \emph{dual} to it.
 A \emph{bigon} is a pair of dual curves that cross at their first and last midcubes. A \emph{monogon} is a single dual curve that crosses itself at its first and last midcubes. A \emph{nonogon} is a single dual curve of length $\geq1$ that starts and ends on the same dual 1-cell, thus it corresponds to an immersed cycle of midcubes. A \emph{spur} is a vertex of degree $1$ on $\partial D$.

A \emph{corner} in a diagram $D$ is a vertex $v$ that is an endpoint of consecutive edges $a,b$ on $\partial D$ lying in a square $s$. We will abuse notation and use the term \emph{corner} also to refer to $s$.
 A \emph{cornsquare} -- short for ``generalised corner of a square''-- consists of a square $s$ and dual curves $p, q$ emanating from consecutive edges $a, b$ of $s$ that terminate on consecutive edges $a',b'$ of $\partial D$. The \emph{outerpath} of the cornsquare is the path $a'b'$ on $\partial D$. Note that, in particular, a corner is a cornsquare.
 A \emph{cancellable pair} in $D$ is a pair of $2$-cells $R_1, R_2$ meeting along a path $e$ such that the following diagram commutes:
 
 \[\begin{tikzcd}
                         & e \arrow[ld] \arrow[rd] &                          \\
\partial R_1 \arrow[rd] \arrow[rr]  &                         & \partial R_2 \arrow[ld] \\
                         & X                       &                         
\end{tikzcd}\]

Given a disc diagram $D$, a cancellable pair  in $D$ leads to a smaller area disc diagram via the following procedure: cut out $e\cup Int(R_1)\cup Int(R_2)$ and then glue together the paths $\partial R_1-e$ and $\partial R_2-e$ to obtain a diagram $D'$ with $\area(D')=\area(D)-2$ and $\partial D'=\partial D$.
\end{definition}

\begin{lemma}\label{lem:discpatho}\cite[2.3+2.4]{WiseIsraelHierarchy}	Let $D \rightarrow X$ be a disc diagram in a non-positively curved cube complex. If $D$ contains a bigon or a nonogon, then there is a new diagram $D'$ having the same boundary path as $D$, so $\partial D' \rightarrow X$ equals $\partial D \rightarrow X$, and such that $Area(D')\leq Area(D)-2$.
Moreover,  no disc diagram in $X$ contains a monogon, and if $D$ has minimal area among all diagrams with boundary path $\partial D$, then $D$ cannot contain a bigon nor a nonogon.
\end{lemma}

\begin{definition}A pair of cone cells $C,C'$ in $D$ is \emph{combinable} if they  map to the same cone $Y$ of $X^*$ and $\partial C$ and $\partial C'$ both pass through a vertex $v$ of $D$, and  map to closed paths
at the same point of $Y$ when regarding $v$ as their basepoint.

As the name suggests, such a pair can be combined to simplify the diagram by replacing the pair with a single cone-cell mapping to $Y$ and whose boundary is the concatenation $\partial C \partial C'$.
\end{definition}

\begin{definition}\label{def:reduced} A \emph{disc diagram} $D\rightarrow X^*$ is \emph{reduced} if the following conditions hold:
\begin{enumerate}
\item \label{it:r1} There is no bigon in a square subdiagram of $D$.
\item \label{it:r2} There is no cornsquare whose outerpath lies on a cone-cell of $D$. 
\item \label{it:r3} There does not exist a cancellable pair of squares. 
\item \label{it:r4} There is no square $s$ in $D$ with an edge on a cone-cell $C$ mapping to the
cone $Y$, such that $(C\cup s)\rightarrow X$ factors as $(C\cup s) \rightarrow Y \rightarrow X$.
\item \label{it:r5} For each internal cone-cell $C$ of $D$ mapping to a cone $Y$, the path $\partial C$ is essential in $Y$.
\item \label{it:r6} There does not exist a pair of combinable cone-cells in $D$.
\end{enumerate}
\end{definition}

\begin{definition}\label{def:mincomp}
The \emph{complexity} $Comp(D)$ of a disc diagram $D\rightarrow X^*$ is the ordered pair $(\#\text{Cone-cells}, \#\text{Squares})$.
We order the pairs lexicographically: namely $(\#C,\#S) < (\#C',\#S')$ whenever $\#C<\#C'$ or $\#C=\#C'$ and $\#S < \#S'$.
A disc diagram $D\rightarrow X^*$ has \emph{minimal complexity} if no disc diagram $D'\rightarrow X^*$ having $\partial D =\partial D'$ has $Comp(D')< Comp(D)$.
\end{definition}

Definitions~\ref{def:reduced} and~\ref{def:mincomp} also make sense for annular diagrams and Möbius diagrams. Note also that if a disc (resp. annular, resp. Möbius) diagram has minimal complexity, then it is also reduced, as any of the pathologies in Definition~\ref{def:reduced} would indicate a possible complexity reduction (see~\cite[3.e]{WiseIsraelHierarchy}). The converse is not necessarily the case.

\begin{definition}[Shell]
A \emph{shell} of $D$ is a 2-cell $C \rightarrow D$
whose boundary path $\partial C \rightarrow D$ is a concatenation $QP_1 \cdots P_k$ for some $k \leq 4$
where $Q$ is a boundary arc in $D$ and $P_1, \ldots , P_k$ are non-trivial pieces in the interior of $D$.
The arc $Q$ is the \emph{outerpath} of $C$ and the concatenation $S:=P_1 \cdots P_k$ is the \emph{innerpath} of $C$.  
\end{definition}

\begin{remark}
Note that if a cubical presentation $X^*$ satisfies the $C(p)$ condition and $D$ is a minimal complexity diagram, then the outerpath of a shell in $D$ is the concatenation of $\geq p-4$ pieces.
\end{remark}

A \emph{pseudo-grid} between $\nu$ and $\mu$ is a square disc diagram $E$ where $\partial  E=\nu \rho \mu^{-1}\varrho^{-1}$ such that each dual curve starting on $\nu$ ends on $\mu$, and vice versa, and where no dual curves starting on $\nu$ cross each other. A pseudo-grid is a \emph{grid} if there are no cornsquares in $E$ with outerpath on $\rho$ nor $ \varrho$.

\begin{definition}[Ladder]
A \emph{ladder} is a disc diagram $L$ admitting an ordering of $n > 2$ cone-cells
and/or vertices $C_1, C_2, \ldots , C_n$ and (possibly trivial) pseudo-grids joining them, where all the cone-cells in $L$ are ordered, and such that, for every $1 < i < n$, the closure of $C_i$ separates $C_{i-1}$ from $C_{i+1}$,  so that:
\begin{enumerate}
\item $\partial L$ is a concatenation $P_1 P_2$ where the initial and terminal points of $P_1$
lie on $C_1$ and $C_n$ , respectively.
\item $P_1 = \alpha_1 \rho_1 \alpha_2 \rho_2 \cdots \alpha_n$ and $P_2 = \beta_1 \varrho_1 \beta_2 \varrho_2  \cdots \beta_n$. 
\item $\partial C_i = \mu_i \alpha_i \nu_i\beta_i$ for each $i$, where $\mu_1$ and $\nu_n$ are trivial paths.
\item $\partial E_i = \nu_i \rho_i \mu^{-1}_{i+1} \varrho^{-1}_i$  for $1 \leq i < n$ and $E_i$ is a pseudo-grid from $\nu_i$ to $\mu_{i+1}$.
\end{enumerate}
Note that we allow the case of a square ladder, where $L$ is necessarily a single pseudogrid.

An \emph{annuladder} $A$ (resp. \emph{Möbiusladder} $M$) is an annular (resp. Möbius) diagram which also has the structure of a ladder, with the slight difference that we require that there be a cyclic ordering on $C_1, C_2, \ldots , C_n$, rather than a linear one.
In other words, the diagram $A$ (resp., $M$)  can be obtained from a ladder $L$ by gluing  $ C_1$ to  $ C_n$ along  (possibly trivial) paths $\sigma_1 \rightarrow \partial C_1, \sigma_n \rightarrow \partial C_n$.
\end{definition}

The main results of classical small-cancellation theory, such as Greendlinger's Lemma and the Ladder Theorem, have suitable analogues in the cubical setting: the cubical $C'(\frac{1}{14})$ versions were proven in~\cite{WiseIsraelHierarchy}; the versions we use are from~\cite{JankiewiczSmallCancellation}.

\begin{theorem}[Diagram Trichotomy/the Cubical Greendlinger Lemma]\label{thm:tric}
Let $X^*= \langle X\mid Y_1, \dots, Y_s\rangle$ be a cubical presentation satisfying the $C(9)$ condition, and let $D\to X^*$ be a disc diagram. Then one of the following holds:
\begin{itemize}
\item $D$ consists of a single cone-cell,
\item $D$ is a ladder, or
\item $D$ has at least $3$ shells and/or cornsquares and/or spurs. Moreover, if $D$ has no shells or spurs, then it has at least $4$ cornsquares.
\end{itemize}
\end{theorem}

\begin{theorem}[Ladder Theorem]\label{thm:ladder}
Let $X^*= \langle X\mid Y_1, \dots, Y_s\rangle$ be a cubical presentation satisfying the $C(9)$ condition, and let $D\to X^*$ be a reduced disc diagram in $X^*$. If $D$ has exactly two shells, then $D$ is a ladder.
\end{theorem}

The Greendlinger Lemma implies elevations of cones are embeddings; we will use this fact extensively in what follows.

\begin{theorem}\label{thm:embeds}
Let $X^*=\langle X\mid \{Y_i\}\rangle$ be a cubical presentation satisfying the $C(9)$ condition. Then each elevation of $Y_i \rightarrow X$ to  $\hat X \rightarrow X$ is an embedding.
\end{theorem}

\subsection{Cohomology and the Cohen-Lyndon property}\label{subsec: CL}

The Cohen-Lyndon property was already introduced in the statement of Theorem~\ref{thm:clintro}. For the sake of exposition, we now state it in general: 

\begin{definition}[The Cohen-Lyndon property for pairs]
Let $G$ be a group, and $\{H_i\}_{i \in I}$ a family of subgroups of $G$. Let $\mathcal{N}_G(H_i)$ denote the normaliser of $H_i$ in $G$.
 The pair $(G,\{H_i\})$ satisfies the \emph{Cohen-Lyndon property} if for each $i \in I$ there exists a left transversal $T_i$ of $\mathcal{N}_G(H_i)\langle \langle \cup_{i \in I} H_i \rangle \rangle$ in $G$ such that $\langle \langle \cup_{i \in I} H_i \rangle \rangle$ is the free product of the subgroups $H_i^t$ for $t \in T_i$. In symbols, $$\langle \langle \cup_{i \in I} H_i \rangle \rangle = \ast_{i \in I, t \in T_i}H_i^t.$$
\end{definition}

Informally, the Cohen-Lyndon property encodes when a collection of subgroups $\mathcal{H}$ and their conjugates are ``as independent as possible''. The obvious way to phrase this is to ask for the normal closure $\langle \langle \mathcal{H}\rangle \rangle$ of $\mathcal{H}$ to be a free product of conjugates of the $H_i \in \mathcal{H}$. In~\cite{KS72},  the Cohen-Lyndon property is called the ``fpmmc property'' which stands for \emph{free product of maximally many conjugates}. This is perhaps more illuminating terminology. 

The Cohen-Lyndon property was first studied in~\cite{CL63}, where  it was defined in the setting of group presentations (i.e., quotients of free groups by normal closures of collections of cyclic subgroups), and shown to hold for one-relator presentations and $C'(\frac{1}{6})$ classical small-cancellation presentations. In the paper~\cite{KS72}, Karrass and Solitar extended the definition to arbitrary quotients, and proved a combination theorem for free products with amalgamation and HNN extensions; this was later elaborated on in~\cite{MCS86}. 
Subsequently, Edjevet and Howie proved in~\cite{EH87} that the Cohen-Lyndon property holds for one-relator quotients of free products of locally indicable groups.  In~\cite{DH94}, the results of Edjvet and Howie were extended to the setting of one-relator products of arbitrary groups, but with extra assumptions on the relator. A rather significant generalisation was obtained by Sun in~\cite{Sun20}, where it was proven that \emph{triples} $(G,\{H_i\},\{N_i\})$ have the Cohen-Lyndon property when the $H_i$ are ``hyperbolically embedded" subgroups of $G$ and the $N_i$ avoids a finite set of ``bad" elements depending only on the $H_i$. In this form, the  $N_i$ are required to be normal subgroups of the corresponding $H_i$, and the $H_i$ replace the normalisers. In~\cite{arenas2023cohenlyndon}, the author extends Cohen and Lyndon's original result for metric $C'(\frac{1}{6})$ classical small-cancellation presentations to non-metric $C(6)$ small-cancellation presentations. Finally,  in~\cite{GHMOSW24}, versions of the Cohen-Lyndon property are proven in the context of drillings of hyperbolic groups.

The Cohen-Lyndon property is useful for many reasons. In the context of this paper, it will play a key role in the proof of Theorem~\ref{thm:asph}. Via the work of Petrosyan--Sun, it will also allow us to understand quite precisely the cohomology of   cubical $C(9)$ small-cancellation quotients. In particular, we use the following:

\begin{proposition}\label{prop:petsun1}\cite[4.6]{PetSun21}
Let $(G, \mathcal{H})$ have the Cohen-Lyndon property, where $\mathcal{H}=\{H_i\}$ let $\Gamma=G/ \nclose{\mathcal{H}}$ and let $A$ be a $\Gamma$-module.
\begin{enumerate}
\item Suppose that for some $p \in \naturals$, $\bigoplus_i H_p(\mathcal{N}(H_i);A)=0$ and that the natural map $H_{p-1}(\mathcal{N}(H_i);A) \rightarrow H_{p-1}(G;A)$ is injective, then 
$$H_p(\Gamma, A)=H_p(G, A)\oplus \bigoplus_i H_p(\mathcal{N}(H_i)/H_i, A).$$ 
\item Suppose that for some $p \in \naturals$, $\bigoplus_i H^p(\mathcal{N}(H_i);A)=0$ and that the natural map $H^{p-1}(G;A) \rightarrow H^{p-1}(\mathcal{N}(H_i);A)$ is surjective, then
$$H^p(\Gamma, A)=H^p(G, A)\oplus \bigoplus_i H^p(\mathcal{N}(H_i)/H_i, A).$$
\end{enumerate}
\end{proposition}

\section{Disassembling the cubical part of $\widetilde{X^*}$}\label{sec:dis}

We need a few definitions and preliminary observations.

A subcomplex $Y$ of a non-positively curved cube complex $X$ is \emph{locally convex} if for each $n$-cube $c$ with $n \geq 2$ in $X$, whenever a corner of $c$
lies in $Y$, then  $c$ also lies in $Y$. When $X$ is CAT(0), local convexity coincides with convexity in the combinatorial metric. However since we will be working with a non-positively curved cube complex that is generally not simply connected, it is necessary to distinguish between the two notions.

We are interested in local convexity because a locally convex subcomplex is, by definition, non-positively curved, and thus aspherical (asphericity is a well-known consequence~\cite{BH99} of the CAT(0) metric on $\widetilde X$, and can also be deduced  combinatorially as in~\cite{Arenas2023thesis}). This fact plays an important role in the rest of this section and in the proof of Theorem~\ref{thm:clprop}.

Let $X^*= \langle X \mid \{Y_i\}_{i \in I} \rangle$ be a cubical presentation and consider the universal cover $\widetilde{X^*}$ of the coned-off space. 
Note that $X$ is a subspace of $X^*$ and thus the preimage $\hat X$ of $X$ in $\widetilde{X^*}$ is a covering space of $X$, namely the regular cover corresponding to $ker(\pi_1X \rightarrow \pi_1X^*)$.  The universal cover of $X^*$ then decomposes as 
\begin{equation*}\label{eq:decomp}
\widetilde{X^*}=\hat X \cup \bigcup_{gY_i} Cone(gY_i)
\end{equation*}
where the $gY_i$ are elevations of $Y_i$'s to $\widetilde{X^*}$, varying over all $i \in I$ and over all left cosets of $Stab_{\pi_1X}(\widetilde Y_i)\nclose{\pi_1 Y_i}$ in $\pi_1X$. We call $\hat X$ the \emph{cubical part} of $\widetilde{X^*}$. 

The reader should think of  $ X^*$ as a ``generalised presentation complex'' for the cubical presentation $\langle X \mid \{Y_i\}_{i \in I} \rangle$, and of $\widetilde X^*$ and $\hat X$ as a generalised Cayley complex and generalised Cayley graph, respectively.

\begin{definition}[Untethered hull]
The \emph{untethered hull} of $\hat X$ is the subcomplex $\mathcal{F} \subset \hat X$ containing all $n$-cubes with $n \geq 1$ that do not lie in any rectangle $R$ intersecting a piece of $\widetilde X^*$ non-trivially. An \emph{untethered component} $\mathcal{F}_\iota \subset \mathcal{F}$ is a connected component of the untethered hull.
\end{definition}

A quick remark is in order.

\begin{proposition}\label{prop: convex free hull}
Each untethered component of $\hat X$ is locally convex.
\end{proposition}

\begin{proof}
Let $\mathcal{F}_\iota\subset \mathcal{F}$ be an untethered component. Let $k$ be the corner of an $n$-cube $c$. If $k$ lies in  $\mathcal{F}_\iota$ but $c$ does not, then $c$ must be contained in a rectangle $R$ which intersects a piece of $\widetilde X^*$ non-trivially. But then the edges of $k$ also lie in $R$, contradicting the definition.
\end{proof}

\begin{definition} A \emph{supporting hyperplane} is a hyperplane $H \rightarrow \widetilde X^*$ whose image is neither contained in a $gY_i$ nor in the untethered hull of $\hat{X}$. A \emph{supporting hyperplane carrier} is a carrier $N(H)$ where $H \rightarrow \widetilde X^*$ is  a supporting hyperplane.
\end{definition}

Consider the following collections of subcomplexes of $\hat X$ \begin{enumerate}
\item \label{it: threeways1} The  translates $\{gY_i\}$ of the relators $Y_1, \ldots, Y_k$ in $\hat X$. 
\item \label{it: threeways2} The supporting hyperplane carriers $\{N(H) \mid H \rightarrow \widetilde X^* \text{ is a supporting hyperplane}\}$.
\item \label{it: threeways3} The untethered components of $\hat X$.
\end{enumerate}  

 The \emph{three-way decomposition} of  $\hat X$ is the set  $\mathcal{U}$ consisting of all subcomplexes of $\hat X$ of type~\eqref{it: threeways1},~\eqref{it: threeways2}, or~\eqref{it: threeways3}.
 In the next subsection, we analyse the connectedness properties of intersections of elements in $\mathcal{U}$. 

\subsection{Intersections, intersections, intersections}\label{sec:inter}

We start by noting that local convexity passes to intersections:

\begin{lemma}\label{lem:loc convex intersections}
If the intersection of locally convex cube complexes is non-empty, then it is a locally convex cube complex.
\end{lemma}

\begin{proof}
Let $Y_1, \ldots, Y_n$ be locally convex subcomplexes of $X$ with $\bigcap_i Y_i \neq \emptyset$
and let $v$ be a vertex in $\bigcap_i Y_i$. If $\link|_{\bigcap_i Y_i}(v)$ contains the boundary of a simplex (i.e., a corner of a cube), then so do $\link|_{Y_1}(v), \ldots,\link|_{Y_n}(v)$. By local convexity, every  simplex  in $\link(v) \subset X$ with that boundary is also contained in $\link|_{Y_1}(v), \ldots,\link|_{Y_n}(v)$, thus all such simplices, and thus the corresponding cubes, also lie in the total intersection. 
\end{proof}

\begin{remark}
In this paper, a  space is said to be \emph{connected} (resp., \emph{path-connected}) if it has \textbf{exactly} one component (resp., path-component). In particular, the empty set in not connected.
\end{remark}

The next order of business is to check that the intersections between elements in $\mathcal{U}$ are simply-connected (and therefore contractible by Lemma~\ref{lem:loc convex intersections}). Most of the proofs follow the same scheme, and rely heavily on Greendlinger's Lemma (Theorem~\ref{thm:tric}).

\begin{proposition}\label{prop: contractible free hull}
Each untethered component of $\hat X$ is simply connected.
\end{proposition}

\begin{proof}
Let $\mathcal{F}_\iota\subset \mathcal{F}$ be an untethered component. 
To check that $\mathcal{F}_\iota$ must be simply connected, let $\sigma \rightarrow \mathcal{F}_\iota$ be a closed path, and assume moreover that $\sigma$ is the shortest path in its homotopy class. Then $\sigma$ bounds a disc diagram $D$ in  $\widetilde X^*$, which can be assumed to have minimal complexity amongst all diagrams with boundary path $\sigma$. By Greendlinger's Lemma, $D$ is either a single vertex or cone cell, a ladder, or has at least $3$ shells, corners, or spurs. Note that $D$ cannot have any shells, since for a shell $C$ the intersection $C \cap \partial D \subset \mathcal{F}_\iota$ contradicts the definition of $\mathcal{F}_\iota$. Since $\sigma$ is the shortest path in its homotopy class, then $D$ has no spurs.

Thus, $D$ must have at least $2$ corners. Each corner corresponds to a square $s$ that must lie in $\mathcal{F}_\iota$ because of local convexity, so $\sigma$ can be pushed across $s$ to obtain a new path $\sigma'$ that is homotopic to $\sigma$ in $\mathcal{F}_\iota$ and bounds a smaller area disc diagram, contradicting our initial choice. Thus $D$, and therefore $\sigma$, must be a single vertex.
\end{proof}

The next result is essentially proven in~\cite[5.7]{WiseIsraelHierarchy}, albeit under  more restrictive hypotheses. We include a proof for the sake of completeness.

\begin{proposition}\label{prop: conecell ints}
Let $gY_i\neq g'Y_{i'}$, then $gY_i\cap g'Y_{i'}$ is either empty or simply connected.
\end{proposition}

\begin{proof}
To show that $gY_i\cap g'Y_{i'}$ is connected, assume towards a contradiction that there exist vertices $x,y$ in distinct connected components of $gY_i\cap g'Y_{i'}$. Amongst all such pairs of vertices, choose $x,y$ so that $\dist(x,y)$ is minimised. Let $\sigma \rightarrow gY_i$ and $\gamma \rightarrow g'Y_{i'}$ be geodesic paths connecting $x$ and $y$. The concatenation $\sigma\gamma^{-1}$ bounds a disc diagram $D$ in $\widetilde X^*$, which can be assumed to have minimal complexity amongst all possible choices of diagrams with boundary path $\sigma\gamma^{-1}$. By Greendlinger's Lemma, $D$ is either a single vertex or cone-cell, a ladder, or has at least $3$ shells, cornsquares, or spurs. Now, the outerpath $\mathcal{O}$ of a shell $S$ in $D$ would either be entirely contained in $\sigma$ or $\gamma$, or would contain $x$ or $y$ and intersect both $\sigma$ and $\gamma$. In either case, $\mathcal{O}$ is the concatenation of at most $2$ pieces, contradicting the $C(9)$ condition. Thus, $D$ cannot have any shells. Moreover, a spur on $\partial D$ would either contradict that $\sigma$ and $\gamma$ are geodesics, or that $x$ and $y$ where chosen to minimise $\dist(x,y)$. Hence $D$ must have at least $4$ cornsquares: one of these cornsquares might contain $x$ and another one might contain $y$, but at least $2$ must have their outerpath lying entirely  in either. After shuffling, such cornsquares yield corners on  $\sigma$ and/or $\gamma$, which can be absorbed into the corresponding cone  $gY_i$ or $g'Y_{i'}$ to obtain paths $\sigma' $ and $  \gamma'$ and a new diagram $D'$ with $\partial D'=\sigma'(\gamma')^{-1}$ whose complexity is strictly less than that of $D$, contradicting our initial hypothesis. Thus $D$ cannot have any shells, cornsquares, or spurs, and must therefore be a single point, i.e., $x=y$.

To show that $gY_i\cap g'Y_{i'}$ is simply connected, consider a closed path $\sigma \rightarrow gY_i\cap g'Y_{i'}$. Then, by definition, $\sigma$ is a piece, and therefore must be nullhomotopic, as otherwise $\sigma$ would contradict the $C(9)$ condition.
\end{proof}

\begin{figure}%[h!]
\centerline{\includegraphics[scale=0.6]{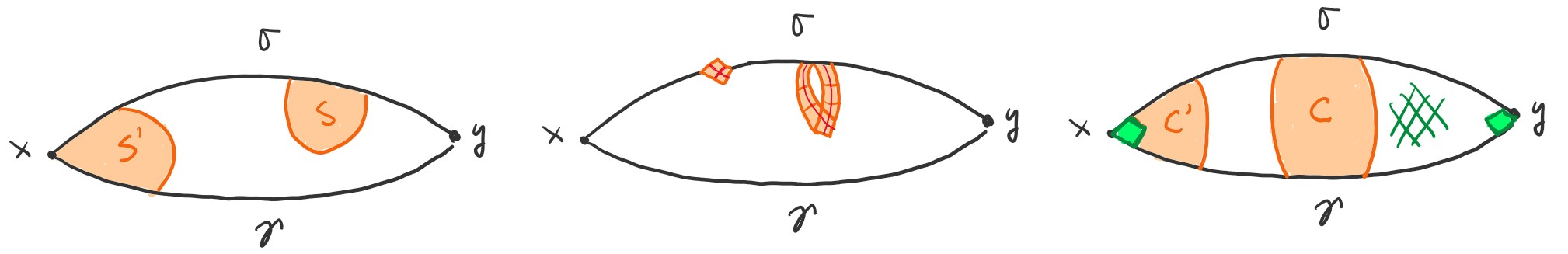}}
\caption{Impossible diagrammatic features in the proof of connectedness in Proposition~\ref{prop: conecell ints}.}
\label{fig:connected}
\end{figure} 

\begin{proposition}\label{prop: diff ints 1}
For each untethered component $\mathcal{F}_\iota\subset \mathcal{F}$, and for each $gY_i$, the intersection $\mathcal{F}_\iota \cap gY_i $ is either empty or simply connected.
\end{proposition}

\begin{proof}
By definition, $\mathcal{F}_\iota \cap gY_i $ cannot contain any edges, so it suffices to show that $\mathcal{F}_\iota \cap gY_i $ is a single vertex.

Let $x,y \in \mathcal{F}_\iota \cap gY_i $ and consider paths $\sigma \rightarrow \mathcal{F}_\iota$ and $\gamma \rightarrow gY_i$ with endpoints $x,y$.
The concatenation $\sigma\gamma^{-1}$ bounds a disc diagram in  $\widetilde X^*$. Choose $\sigma, \gamma, D$ so that $D$ has minimal possible complexity.
By Greendlinger's Lemma, $D$ is either a single vertex or cone cell, a ladder, or has at least $3$ shells, cornsquares, or spurs. By the minimality of our choices, $D$ can have no spurs.
Suppose there is a shell $S$  and let $\mathcal{O}$ be its outerpath. Note that $\mathcal{O}$ must lie on $\gamma$, as the outerpath of a shell lying on $\sigma$ would imply that $\mathcal{F}_\iota$ contains an edge lying in a piece.  Since $\mathcal{O}$ is a subpath of $\gamma$, then $\mathcal{O}$ lies in the intersection of $2$ cone cells -- namely $S$ and a cone-cell mapping to $gY_i$, and is therefore a piece, contradicting the $C(9)$ condition.
Hence, $D$ must have corners. By local convexity of $\mathcal{F}_\iota$ and $gY_i$, a corner on $\gamma$ could be pushed into $\mathcal{F}_\iota$, and a corner on $\gamma$ could be pushed into $gY_i$. Thus,  $D$ can have at most $2$ corners, each containing an edge of $\sigma$ and an edge of $\gamma$. Hence  $D$ is a square ladder by the Ladder Theorem (Theorem~\ref{thm:ladder}). But a square ladder lies in a hyperplane carrier, and therefore yields a piece between $gY_i$ and $\mathcal{F}_\iota$, contradicting the definition of $\mathcal{F}_\iota$.

We conclude that $D$ is a single vertex, so $x=y$. 
\end{proof}

\begin{proposition}\label{prop:different ints  3} 
For each untethered component $\mathcal{F}_\iota\subset \mathcal{F}$, and for each supporting hyperplane carrier $N(H)$, the intersection $N(H)\cap \mathcal{F}_\iota$
is either empty or simply connected.
\end{proposition}

\begin{proof}
The proof is similar to that of Proposition~\ref{prop: diff ints 1}, so we repeat the argument in a little less detail: to prove the assertion, it suffices to show that if $\mathcal{F}_\iota \cap N(H)$ is non-empty, then it is a single vertex.

Let $x,y \in \mathcal{F}_\iota \cap N(H) $ and consider paths $\sigma \rightarrow \mathcal{F}_\iota$ and $\gamma \rightarrow N(H)$ with endpoints $x,y$.
The concatenation $\sigma\gamma^{-1}$ bounds a disc diagram in  $\widetilde X^*$, and $\sigma, \gamma$ and $ D$ can be chosen so that $\dist (x,y)$ is minimal amongst all pairs of vertices in $\mathcal{F}_\iota \cap N(H)$ and so that $D$ has minimal possible complexity.
We see immediately that $D$ cannot have any shells, as a shell on either  $\sigma$ or $\gamma$ would contradict the $C(9)$ condition, and a shell intersecting both $\sigma$ and $\gamma$ would contradict the definition of $\mathcal{F}_\iota$. Thus,  Greendlinger's Lemma and the fact that  $\dist (x,y)$ is minimal, imply that $D$ is either a single vertex  or has at least $4$ cornsquares, at least $2$ of which lie entirely in $\gamma$ or $\sigma$. But such consquares can be absorbed into $\mathcal{F}_\iota$   or $N(H)$, contradicting the minimal complexity of $D$. Thus, $D$ is a single vertex and $x=y$.
\end{proof}

Similarly:

\begin{proposition}\label{prop:different ints 2} 
For each  $gY_i$ and for each supporting hyperplane carrier $N(H)$, the intersection 
$N(H)\cap gY_i$ 
is  either empty or simply connected.
\end{proposition}

\begin{proof}
First we show that $N(H)\cap gY_i $ is connected. Let $x,y \in N(H)\cap gY_i $, and assume that $x,y$ lie in distinct connected components. Assume moreover that $\dist(x,y)$ is the least possible amongst all such $x,y$. Let $\sigma \rightarrow N(H)$ and $\gamma \rightarrow gY_i$ be geodesic paths with endpoints $x,y$. Then $\sigma\gamma^{-1}$ bounds a disc diagram $D$. Choose $\sigma, \gamma$, and $D$  as above to minimise the complexity of $D$. By Greendlinger's Lemma, $D$ is either a single vertex or cone-cell, a ladder, or has at least $3$ cornsquares, shells, or spurs. Note that $D$ cannot have  spurs or shells on neither $\sigma$ nor $\gamma$. Indeed, spurs contradict that $\sigma$ and $\gamma$ are geodesic paths, and   a shell would contradict the $C(9)$ condition, as its outerpath would be a single piece. For the same reason, $D$ cannot be a single cone-cell.

Since $D$ has no shells or spurs, it must have at least $4$ cornsquares. But any cornsquare on either $\sigma$ or $\gamma$ could be absorbed into either $N(H)$ or  $gY_i $ to reduce the complexity of $D$, contradicting the minimality of the choice of $D$. So $D$ can have at most $2$ cornsquares, again leading to a contradiction.
 Hence $D$ is a single vertex, so $x=y$, contradicting our initial hypothesis.

Now it is easy to see that $N(H)\cap gY_i$ is simply connected. Indeed, a path $\sigma \rightarrow N(H)\cap gY_i$ is a piece by definition, and thus,  by the $C(9)$ condition, it cannot be essential.
\end{proof}

The next result shows that supporting hyperplanes in $C(9)$ cubical presentations behave very much like hyperplanes in CAT(0) cube complexes, inasmuch as their embedding properties are concerned.

\begin{lemma}\label{lem:special like behaviours}
In $\widetilde X^*$:
\begin{enumerate}
\item \label{it:0} each supporting hyperplane carrier is simply connected,
\item  \label{it:1} each  supporting hyperplane carrier is embedded, and
\item  \label{it:2} the intersection between two supporting hyperplane carriers is either empty or simply connected.
\end{enumerate}
\end{lemma}

\begin{proof}
For~\eqref{it:0}, suppose $N(H) \rightarrow \widetilde X^*$ is not simply connected. A non nullhomotopic loop $\sigma \rightarrow H$ gives rise to a disc diagram $D$ in $\widetilde X^*$ whose boundary path is a closed non-nullhomotopic path in $N(H)$. Assume that $D$ has minimal possible complexity amongst all  choices of $\sigma$ and  $D$. 
By Greendlinger's Lemma, $D$ is either a single cone cell or 0-cell, a ladder, or has at least 3 shells, cornsquares, or spurs. Any shell in $D$ would contradict the $C(9)$ condition, and since hyperplanes in CAT(0) cube complexes are simply connected, $D$ cannot be a square disc diagram. 
Thus, $D$ is either a 0-cell or has at least 2 cornsquares (and some internal cone-cells). After shuffling, a cornsquare on $\partial D$ corresponds to a square that can be pushed out of $D$ to obtain a new diagram $D'$ with $\area D' < \area D$ and such that $\partial D$ is still a path in the carrier of $N(H)$, thus contradicting our initial choice. We conclude that $D$ is a single 0-cell, which again leads to a contradiction, since we had assumed that $\sigma \rightarrow H$ was not null-homotopic in $N(H)$.

For~\eqref{it:1}, let $N(H) \rightarrow \widetilde X^*$ be a non-embedded hyperplane carrier, so $N(H)$ self-osculates at a vertex $p$ (note that $H$ itself might not self-intersect). Let $\sigma \rightarrow \hat{X}$ be a closed path in the image of $N(H)$ and based at $p$. Then $\sigma$ bounds a disc diagram $D$ in $\widetilde X^*$. Amongst all possible choices of $\sigma$ as above and of $D$ with $\partial D =\sigma$, choose $(\sigma, D)$ so $\sigma$ is the shortest possible and $D$ has minimal complexity. 
By Greendlinger's Lemma, $D$ is either a single cone cell or 0-cell, a ladder, or has at least 3 shells, cornsquares, or spurs. Any shell in $D$ would contradict the $C(9)$ condition since its outerpath would be a single piece, and since hyperplanes in CAT(0) cube complexes are embedded, $D$ cannot be a square disc diagram, so $D$ is either a 0-cell or has at least 2 cornsquares (and again, contains some internal cone-cells). One of these cornsquares might contain the basepoint $p$, but the other one does not, so we  can now conclude as in the proof of~\eqref{it:0} by first shuffling the diagram to find a corner on $\partial D$ and then pushing out a square at the corner not containing $p$.

For~\eqref{it:2}, let $N(H) \rightarrow \widetilde X^*$ and let $N(H') \rightarrow \widetilde X^*$ be a pair of hyperplane carriers whose intersection is non-empty. 
Let $x,y \in N(H) \cap N(H')$. Consider paths $\sigma \rightarrow N(H),\gamma \rightarrow N(H')$ with endpoints $x,y$. 
The concatenation $\sigma \gamma^{-1}$ bounds a disc diagram $D$ in $\widetilde X^*$. 
Choose $\sigma, \gamma, D$ to minimise complexity of $D$. As before, applying Greendlinger's Lemma and using the $C(9)$ condition, we see that $D$ has no shells and thus must have spurs and/or cornsquares or be a single 0-cell. If $D$ has $\geq 3$ cornsquares or spurs, then at least one of these lies on either $N(H)$ or $N(H')$, and can thus be pushed out or deleted to obtain a diagram of lesser complexity than $D$, as in previous cases. 
Therefore, $D$ must be a ladder with a cornsquares or spur at either end. Suppose $\area D > 0$. If $D$ is a square ladder, then it must be a degenerate ladder (a segment) because  hyperplane carriers in CAT(0) cube complexes have connected intersection. Hence, $D$ has a cone-cell $S$, but then the boundary path of $S$ is a concatenation of at most 6 pieces: one corresponding to each intersection $\partial S \cap N(H), \partial S \cap N(H')$ and at most four (two at either end) corresponding to intersections between $\partial S$ and the previous and subsequent cells in the ordering of the ladder, contradicting the $C(9)$ condition.

 We conclude that $\area (D)=0$ and thus that $\sigma=\gamma^{-1}$, so $x,y$ lie in the same connected component of $N(H) \cap N(H')$. Since this is the case for every $x,y \in N(H) \cap N(H')$, the intersection $N(H) \cap N(H')$ is connected. 
 
 Finally, we show that $N(H) \cap N(H')$ is simply connected. Consider a  loop $\sigma \rightarrow N(H) \cap N(H')$ and assume moreover that $\sigma$ is not nullhomotopic, and is a shortest representative of its free homotopy class. The map  $\sigma \rightarrow N(H) \cap N(H')$ factors through both $N(H)$ and $N(H')$, and both maps $\sigma \rightarrow N(H)$ and $\sigma \rightarrow N(H')$ have to be nullhomotopic by~\eqref{it:0}. In particular, $\sigma$ bounds a minimal area square disc diagram $D$ in $N(H)$. In this case, Greendlinger's Lemma implies that $D$ must have at least 3 cornsquares (there are no spurs because $\partial D =\sigma$ is a shortest path in its homotopy class), which are also cornsquares in $N(H')$. Since $N(H) $ and $ N(H')$ are locally convex, this implies,  after shuffling $D$ to replace a cornsquare with a square corner, that every cube $c$ at such a corner $k$ must also lie in $N(H) \cap N(H')$, and this therefore that the path $\sigma$ can be homotoped across $c$ to reduce the area of $D$, contradicting the minimality of the choice of $D$.
\end{proof}

\begin{figure}%[h!]
\centerline{\includegraphics[scale=0.65]{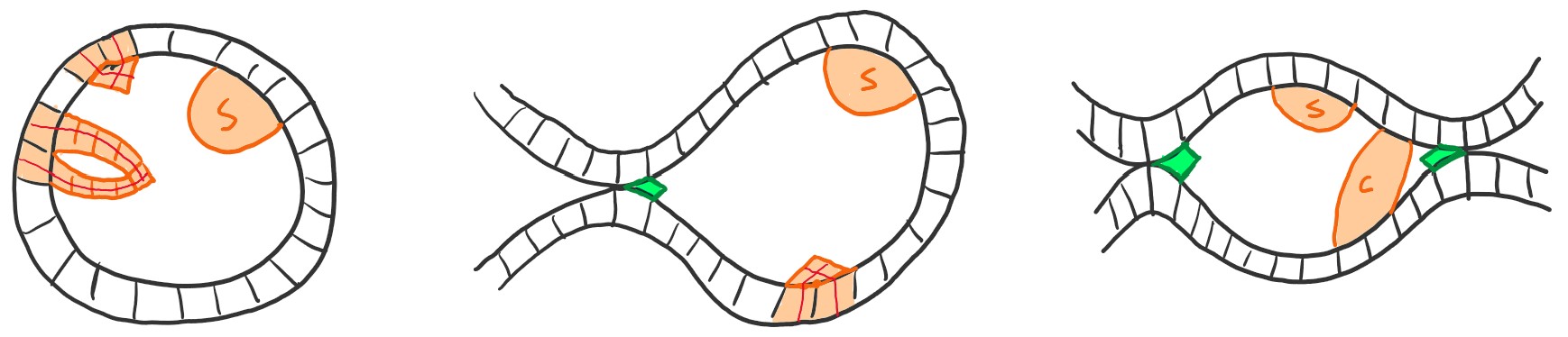}}
\caption{Heuristics of the relevant disc diagrams in the proofs of the various items in Lemma~\ref{lem:special like behaviours}. From left to right: hyperplane carriers are simply-connected, embedded, and their pairwise non-empty intersections are simply-connected. Impossible cornsquares and shells  are indicated in orange; possible cornsquares are marked in green.}
\label{fig:hyps}
\end{figure} 

The previous results will be used extensively in the base cases of various inductive arguments. For ease of reference, we collect them into a single statement:

\begin{theorem}\label{thm:combointersections}
Let $X,X' \in \mathcal{U}$. Then the intersection $X \cap X'$ is either empty or simply connected.
\end{theorem}

To finish this section, we check that a version of the Helly Property  holds for intersections between cones, hyperplane carriers, and untethered components in $\mathcal{U}$. We then strengthen the conclusion to deduce simple-connectivity of non-empty intersections between these objects. These results will be fundamental for the proof of Claim~\ref{clm:final}, which is in turn an essential supporting statement in the proof of Theorem~\ref{thm:clprop}.

\begin{proposition}\label{prop:helly}
Let $X_1, \ldots, X_n \in  \mathcal{U}$. If each pairwise intersection $X_i \cap X_j$ is non-empty, then the total intersection $\bigcap_{i\in \{1, \ldots, n\}} X_i$ is non-empty.
\end{proposition}

\begin{proof}
The proof is by induction on $n$. For the base case, 
we check that $X_1 \cap X_2 \cap X_3$ is non-empty whenever each intersection $X_i \cap X_j$ with $i,j \in \{1,2,3\}$ is non-empty. To this end, let $i <j$,   consider vertices $p_{ij} \subset X_i \cap X_j$ and let $\rho_1, \rho_2, \rho_3$ be geodesics connecting the $p_{ij}$'s.  Let $D$ be a disc diagram with boundary path $\rho_1\rho_2\rho_3$, and assume that $D$ has minimal complexity.  If $D$ is a square disc diagram, then by the Helly property for CAT(0) cube complexes~\cite[2.10]{WiseIsraelHierarchy}, $D$ must be a \emph{tripod} -- i.e., a degenerate disc diagram with a central vertex $v$ and three branches coming off of it. In that case, $v \subset X_1 \cap X_2 \cap X_3$ and the proof is complete. So we may assume that $D$ is not a square disc diagram, and by Greendlinger's lemma, is either a ladder, or has at least 3 shells, cornsquares, or spurs. Any shell would contradict the C(9) condition, since its outerpath would be a concatenation of at most 3 pieces, and any spur in the interior of one of $\rho_1,\rho_2,\rho_3$  would contradict that these are geodesics. Thus, $D$ must have at least $4$  cornsquares on $\partial D$, and at least one of these must lie on a $\rho_i$. But, after shuffling $D$, such a cornsquare can be pushed out of the diagram to produce a diagram of less complexity, contradicting our initial choice of $D$. Thus, $D$ must be a tripod  with a central vertex $v $ that lies on all three paths $\rho_1,\rho_2,\rho_3$. Therefore, $v \subset X_1\cap X_2 \cap X_3$. This finishes the proof of the base case.

Now assume that the total intersection $\bigcap_{i\in \{1, \ldots, n\}} X_i$ is non-empty whenever $n<N$. Let $X_1, \ldots, X_N $ be elements of  $\mathcal{U} $ that pairwise intersect. We can prove that $\bigcap_{i\in \{1, \ldots, N\}} X_i$ is non-empty in the same way as the base case. Define, for each $j,j' \subset \{1, \ldots, N\}$, the following subcomplexes: $$Z_j=\bigcap_{i\in \{1, \ldots, N\}-\{j\}} X_i$$ and  $$Z_{jj'}=\bigcap_{i\in \{1, \ldots, N\}-\{j,j'\}} X_i.$$

Consider geodesic paths $\mu_1 \rightarrow Z_1, \mu_2 \rightarrow Z_2$ and $ \mu_3 \rightarrow Z_3$ with endpoints $m_{12}\subset Z_{12},m_{23} \subset Z_{23}, m_{31} \subset Z_{31}$. 
Let $D$ be a minimal complexity disc diagram bounded by $\mu_1\mu_2\mu_3$. 
Applying Greendlinger's Lemma, the same considerations as in the case $N=3$ imply that $D$ must be a tripod, and the central vertex $v$ lies in   $Z_1, Z_2$ and $Z_3$, and therefore in $\bigcap_{i\in \{1, \ldots, N\}} X_i$.
\end{proof}

\begin{figure}%[h!]
\centerline{\includegraphics[scale=0.65]{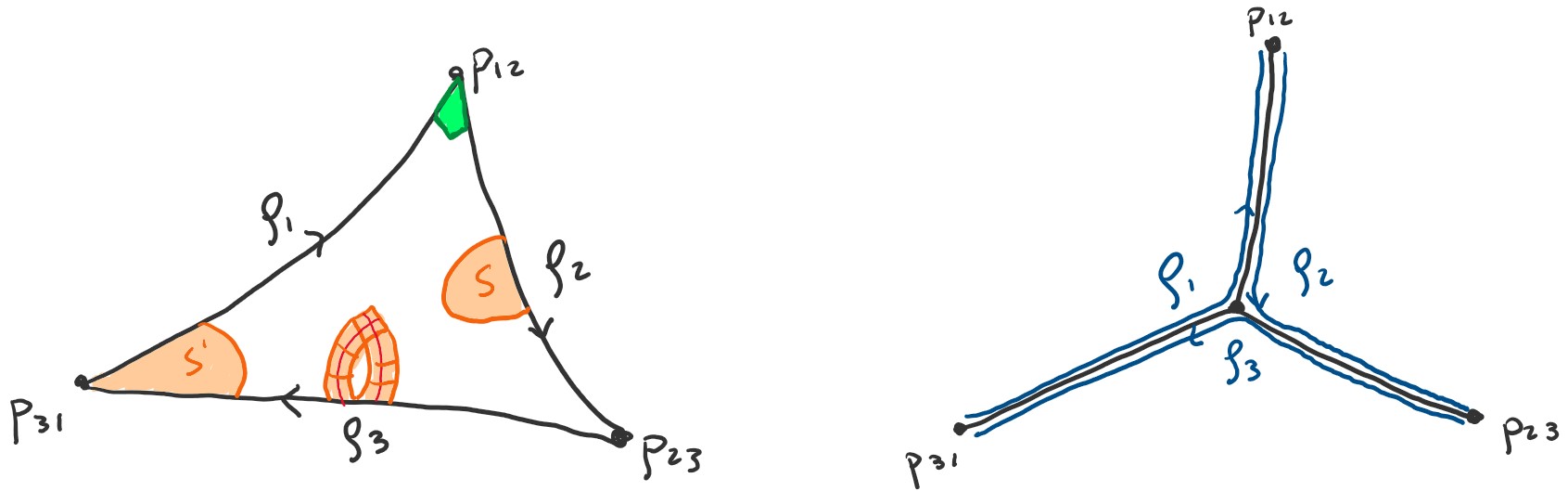}}
\caption{The objects considered in the proof of Proposition~\ref{prop:helly}.}
\label{fig:helly}
\end{figure} 

\begin{proposition}\label{prop:stronghelly}
Let $X_1, \ldots, X_n \in  \mathcal{U}$. If  the total intersection $\bigcap_{i\in \{1, \ldots, n\}} X_i$ is non-empty, then it is simply-connected.
\end{proposition}

\begin{proof} The proof is by induction on $n \in \naturals$. The base case is $n=2$ and follows from Theorem~\ref{thm:combointersections}. 

Assuming that the statement holds whenever $n <N$, consider a collection $X_1, \ldots, X_N \in  \mathcal{U}$. 
By the induction hypothesis, each intersection $\bigcap_{i \in I} X_i$ where $|I|\leq N-1$ is simply-connected.

Connectivity of $\bigcap_{i \in \{1, \ldots, N\}} X_i$  now follows as in the proofs of connectivity in Proposition~\ref{prop: conecell ints} and Proposition~\ref{prop:different ints 2} above, by considering
  vertices $x,y$ in distinct connected components of $\bigcap_{i \in \{1, \ldots, N\}} X_i$ satisfying that $\dist(x,y)$ is minimal possible amongst all such choices,   considering  geodesic paths $\sigma \rightarrow \bigcap_{i \in \{1, \ldots, N-1\}} X_i$ and $ \gamma \rightarrow X_n$  with endpoints $x,y$, and applying Greendlinger's Lemma to a minimal complexity disc diagram bounded by $\sigma\gamma^{-1}$. 
  
  Similarly, simple-connectivity follows as in the proofs of simple-connectivity in Proposition~\ref{prop: conecell ints} and Proposition~\ref{prop:different ints 2}, by considering a closed path $\sigma \rightarrow \bigcap_{i \in \{1, \ldots, N\}} X_i$ that is shortest in its homotopy class. If at least one of the $X_i$'s is a cone,  then $\sigma$ is a piece and therefore cannot be essential. If for all the  $i \in \{1, \ldots, N\}$, the corresponding $X_i$ is a supporting hyperplane carrier or an untethered component, then we arrive at a contradiction by applying Greendlinger's Lemma to a minimal complexity disc diagram bounded by $\sigma$. 
\end{proof}

\section{Reassembling the cubical part of $\widetilde{X^*}$ (carefully)}\label{sec:order}

\begin{definition}[The structure graph]\label{def:structure}
Define the \emph{structure graph} $\Lambda$ of the cubical presentation $X^*=\langle X \mid \{Y_i\} \rangle$ as follows. The vertex set $V(\Lambda):=V$ of $\Lambda$ has three types of vertices: \begin{enumerate}
\item The first type corresponds to the translates $\{gY_i\}$ of the relators $Y_1, \ldots, Y_k$ in $\hat X$. The set of vertices of this type will be denoted $V_C$. 
\item The second type, denoted by $V_H$, corresponds to the supporting hyperplane carriers in $\widetilde X^*$.
\item The third type of vertices, denoted $V_{\mathcal{F}_\iota}$, corresponds to the untethered components of $\hat X$.
\end{enumerate}  

Let $X_v$ denote the subcomplex of $\hat X$ corresponding to $v \in V$. We may write the three-way decomposition of $\widetilde X^*$ as $$\mathcal{U}=\{X_v : v \in V\}.$$

The edges of $\Lambda$   correspond to non-empty intersections of one of the following types:
\begin{enumerate}

\item   $gY_i \cap g'Y_j$, 
\item  $gY_i \cap \mathcal{F}_\iota$,
\item $gY_i \cap N(H)$, 
\item $N(H) \cap N(H')$, 
\item $N(H) \cap  \mathcal{F}_\iota$.
\end{enumerate}
Where  the $gY_i, g'Y_j$ range
over all $i, j \in I$ and over all left cosets of $Stab_{\pi_1X}(\widetilde Y_i)\nclose{\pi_1 Y_i}$ and  $Stab_{\pi_1X}(\widetilde Y_j)\nclose{\pi_1 Y_j}$ in $\pi_1X$, $\iota $ ranges over the connected components of $\mathcal{F}$, and $H, H'$ are supporting hyperplanes. 

Note that each intersection $gY_i \cap g'Y_j$ is a piece $p$ or a vertex, so an edge of  $\Lambda$ having an endpoint in $V_C$ corresponds to a piece in $\hat X$, to a vertex, or to an intersection between a $gY_i$ and an untethered component of $\mathcal{F}$.
\end{definition}

\begin{lemma}\label{lem:simplicial structure graph}
Let $\Lambda$ be the structure graph of a $C(9)$ cubical presentation $X^*$, then $\Lambda$ is a simplicial graph.
\end{lemma}

\begin{proof}
By Theorem~\ref{thm:embeds}, each $Y_i \rightarrow X$ embeds in $\hat{X}$, by Lemma~\ref{lem:special like behaviours}  each hyperplane carrier  $N(H)$ embeds in  $\hat{X}$, and by definition, each component of $\mathcal{F}$ also embeds in $\hat{X}$. Thus $\Lambda$ contains no loops.
By Theorem~\ref{thm:combointersections}, each non-empty intersection $X_v \cap X_u$ is  connected, so  $\Lambda$ has no bigons.
\end{proof}

\begin{lemma}\label{lem:iscover}
 $\mathcal{U}$ is a topological cover of $\hat X$.
\end{lemma}

\begin{proof}
Let $c$ be a cube in $\hat X$. Then either $c$ lies in some $gY_i$, or (and this is not an exclusive ``or") $c$ lies  in the hyperplane carriers corresponding to the hyperplanes dual to edges of $c$.
\end{proof}

\begin{definition}[The nerve of $\mathcal{U}$]
The \emph{nerve complex} $\mathbf{N}(\mathcal{U})$ of a topological covering $\mathcal{U}$ is the abstract simplicial complex $$\mathbf{N}(\mathcal{U})=\{V'  \subset V  : \bigcap_{v \in V'} X_v \neq \emptyset, |V'|< \infty \}.$$
\end{definition}

 We  order  the vertices of $\mathbf{N}(\mathcal{U})$ using the ordering introduced in~\cite{arenas2023cohenlyndon} in the setting of classical small cancellation theory. The ordering is a refinement of the \emph{Lusin–Sierpi\'nski} or \emph{Kleene–Brouwer} order, and plays an essential role in the proof of Theorem~\ref{thm:clprop}.

\begin{definition}[Orderings on $\mathcal{U}$]\label{def:order}
Choose $v_0 \in V(\Lambda)$ corresponding to a $gY_i$. 

We define a total ordering $\leq$ on $\mathcal{U}$, that is, an injective function $\varphi: V \rightarrow \naturals$. 
To do this, we first define an ordering on the simplices of $\mathbf{N}(\mathcal{U})$ inductively as follows.

Start by setting $\varphi(v_0)=0$, and  define 
$$A_0=\{u \in V\ : \{u,v_0\} \in \mathbf{N}(\mathcal{U})\}\cup\{v_0\}.$$ 
Choose $u_1 \in A_1$, let $\varphi(u_1)=1$, and let 
$$A_{01}=\{u \in V : \{v_0, u_1, u\} \in \mathbf{N}(\mathcal{U})\}\cup\{v_0, u_1\}.$$ Inductively, assume that $\varphi$ has been defined for a subset of cardinality $k$, so $\varphi(v_0)=0, \ldots, \varphi(v_k)=k$. 
For each non-empty simplex $\{v_i, \ldots, v_\ell \}$ of $\mathbf{N}(\mathcal{U})$ where $\varphi(v_i), \ldots, \varphi(v_\ell)$ are already defined, let $$A_{v_i\ldots v_\ell}=\{u \in V : \{v_i, \ldots, v_\ell\}\cup \{u\} \in \mathbf{N}(\mathcal{U})\}\cup \{v_i\ldots v_\ell\}.$$ 
We view each simplex $\{v_i, \ldots, v_\ell \}$ as an ordered tuple $(v_i, \ldots, v_\ell )$ where $\varphi(v_i) < \varphi(v_{i+1}) < \ldots < \varphi(v_\ell)$, and 
order the simplices using the Lusin–Sierpi\'nski order, which is defined as follows. For a pair of simplices, set $\{v_i, \ldots, v_\ell \}< \{w_{i'}, \ldots, w_{\ell'} \}$ if either 
\begin{enumerate}
\item  there exists $j \leq \min\{\ell,\ell'\}$ with  $v_\iota=w_\iota$ for all $ \iota < j$, and $v_{j}< w_{j}$, or
\item $\ell > \ell'$ and $j=j'$ for all $j'\leq \ell'$.
\end{enumerate}

Now consider a least simplex $\{v_i\ldots v_\ell\}$  such that there exists  $u \in A_{v_i\ldots v_\ell}$ whose image is not yet defined 
 and choose such a vertex $u$ arbitrarily. 
Set $\varphi(u)=k+1$. 
\end{definition}

In Definition~\ref{def:order}, a number of choices have to be made. Thus we obtain not just one, but in general infinitely many orderings. Moreover, we note that from the definition, it is not immediate that  a $\varphi$ as inductively constructed above is defined on all of $V$. This motivates the following:

\begin{definition}
Let $\Phi$ be the set of all injective functions $\varphi:V' \rightarrow \naturals$ such that $V' \subset V$, and that satisfy the conditions in Definition~\ref{def:order}. An  element of $\Phi$ is an \emph{admissible ordering} if $V'=V$.
\end{definition}

For vertices $v,v' \in V$, define $\dist(v,v')$ as the least number of edges in a path connecting $v$ and $v'$. The following Lemma was proven in~\cite[Lem. 3.11]{arenas2023cohenlyndon}. While the underlying space and the structure graph are different therein, the proof depends only on Definition~\ref{def:order}, thus it still holds in our setting.

\begin{lemma}\label{lem:well-defined} 
Let $v, v' \in V$ and let $\varphi \in \Phi$. If $\varphi(v)<\varphi(v')$, then $\dist(v, v_0) \leq \dist(v', v_0)$. In particular, the function $\varphi: V \rightarrow \naturals$ is well-defined, so every  $ \varphi \in \Phi$ is admissible.
\end{lemma}

\section{The Cohen-Lyndon property for cubical C(9) presentations}\label{sec:CL prop}

\begin{lemma}\label{clm:connected induction} Let $ X^*$ be  a $C(9)$ cubical presentation, let $V=\{v_j\}_{j \in \naturals}$ be the vertices of its structure graph $\Lambda$, and for each $v_j \in V$, let $X_{v_j}$ be the corresponding element of $\mathcal{U}$. For each $k \in \naturals$, the  intersection 
$$\bigcup_{j < k} X_{v_j} \cap X_{v_k}$$
is connected.
\end{lemma}

The main technical ingredient in the proof of Lemma~\ref{clm:connected induction} is Claim~\ref{clm:final} below. 
The proof is a little tedious, but the basic idea  is  to combine an inductive argument on the area of the disc diagram alluded to in the claim with a  case-by-case analysis of how the ordering on the elements of $\mathcal{U}$ interacts with Greendlinger's Lemma.

A \emph{thick annuladder} $\mathbf{A}$ is an annular diagram $A$ which has the structure of an annuladder together with a finite number of squares attached to the transitions between distinct cone-cells and between cone-cells and pseudogrids in the underlying annuladder for $\mathbf{A}$. Such squares are called  \emph{essential corners}. A \emph{thick Möbiusladder} $\mathbf{M}$ is defined similarly.

\begin{figure}[h!]
\centerline{\includegraphics[scale=0.55]{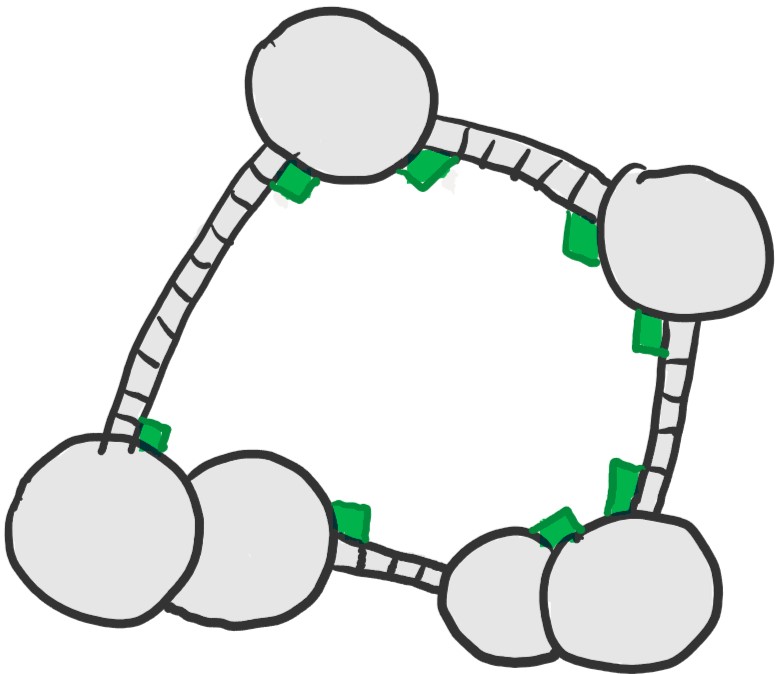}}
\caption{A thick annuladder.}
\label{fig:thick}
\end{figure} 

\begin{claim}\label{clm:final} Let  $ X^*$ be  a $C(9)$ cubical presentation, let $V=\{v_j\}_{j \in \naturals}$ be the vertices of its structure graph $\Lambda$, and for each $v_j \in V$, let $X_{v_j}$ be the corresponding element of $\mathcal{U}$.  Let $\bigcirc_\beta$ be either a  thick annuladder or a thick Möbiusladder in $\widetilde X^*$ collaring a minimal complexity disc diagram $D_\beta$ so that $\partial D_\beta= D_\beta\cap \bigcirc_\beta=\beta$, and finally, let $X_{v_{k_0}}$ be the maximal element in the ordering of Definition~\ref{def:order} intersecting $\bigcirc_\beta$ in a cone-cell or a pseudogrid, and  that does not correspond to an untethered component. Then $K < X_{v_{k_0}}$  for every $K$ satisfying that $K \cap D_\beta$  is a shell or a hyperplane containing a cornsquare  in $D_\beta$. 
\end{claim}

To avoid saying ``let $\bigcirc_\beta$ be a thick annuladder or thick Möbiusladder'' repeatedly throughout the proof,  we prove the claim under the assumption  that $\bigcirc_\beta$ is a thick annuladder $A_\beta$, but note that nothing is being used about $A_\beta$ that does not also hold in the Möbiusladder case.

\begin{proof}
 The proof is by induction on the area of $D$. The base case is $\area(D)=1$, so $D$ is either a single square, or a single cone-cell. If $D$ is a single cone-cell, then $D$ is a shell, and if $D$ is a single square then by definition it is also a cornsquare. 
 Let $K \in \mathcal{U}$ be the element of the three-way-decomposition satisfying $K \cap D=D$. 
 Suppose that $K > X_{v_{k_0}}$, and consider the union $\bigcup_{j < k_0} X_{v_j}$.  
 Since $X_{v_{k_0}}$ is the next element in the ordering, then  there is a  simplex $\sigma$ containing $X_{v_{k_0}}$, and  a subsimplex $\sigma' \subset \sigma$ which is least in the Lusin--Sierpi\'nski ordering, and such that  $v_{k_0} \subset \sigma -\sigma'$ is not in $\bigcup_{j < k_0} X_{v_j}$. 
 Let $X_{v_{j_1}}< \ldots < X_{v_{j_m}} < X_{v_{k_0}}$ be the boundaries of the 2-cells in $A_\beta$.  We now consider two cases: 
 \begin{enumerate}
 \item \label{clmit:1} Either all of the vertices $v_{j_1}, \ldots, v_{j_m}$ lie in $\sigma$, or
 \item \label{clmit:2} some vertex in  $\{v_{j_1}, \ldots, v_{j_m}\}$ does not lie in $\sigma$.
 \end{enumerate}
In the first case,  the  Helly Property proven in Lemma~\ref{prop:helly} together with Lemma~\ref{prop:stronghelly} imply that a boundary component of $A_\beta$ bounds a reduced disc diagram $E$ in $\widetilde X^*$ such that the cubical part of $E$ factors through $\bigcup_{v \in\sigma}X_v$ -- i.e., the squares and boundaries of cone-cells in $E$ lie in $\bigcup_{v \in\sigma}X_v$. 
We may moreover assume that $E$ has minimal complexity amongst all diagrams satisfying this condition. 
We claim that if $\partial D=\partial E$, then $E$ is a single cone-cell or square. Indeed, assume that $\area(E)\geq 2$; then Greendlinger's Lemma implies, in particular, that $E$ contains a shell or a cornsquare. If $D$ is a square and $E$  has a shell, then this contradicts that $D$ is reduced, since $D$ can be absorbed into the shell. Similarly,  if $D$ is a cone-cell and $E$  has a cornsquare, this contradicts instead that $E$ is reduced. 
If $D$ is a square and $E$  has a cornsquare, consider the diagram $E \cup D$ where $E$ is glued to $D$ along the outerpath of the cornsquare. 
Then $E \cup D$ has a bigon, and thus $E \cup D$  is not reduced -- this implies that the pair of squares $E,D$ can be cancelled, and thus that there is a diagram $E'$ with the same boundary path as $E$ and less complexity, contradicting our assumptions on $E$. 
Finally, if $D$ is a cone-cell and $E$  has a shell, then this contradicts the $C(9)$ condition, since then the outerpath of the shell is a single piece. 

Thus, if $\partial D=\partial E$, then $E$ is a single cell. This implies that $K > X_{v_{k_0}}$ since then  $D=E$, as  otherwise the non-positive curvature of  $\hat X$ (when $D$ is a square) or the $C(9)$ condition (when $D$ is a cone-cell) would be violated. 
We now show that $\partial D=\partial E$. 
Assume this is not the case and consider the unions $D \cup A_\beta$ and $E \cup A_\beta$ where, in each case, $D$ or $E$ is glued to $A_\beta$ along its boundary path. 
Observe that if $A_\beta$ is a square annular diagram, then $A_\beta$ has no essential corners, and is therefore an authentic annuladder. In this case, the outerpaths of both $D$ and $E$ would consist of a single piece,  contradicting the $C(9)$ condition.
 We may thus assume that $A_\beta$ has at least one cone-cell.
Let $C$ be a cone-cell on $A_\beta$, and note that $C$ has to intersect both $D$ and $E$ in non-trivial paths. 
 Then the boundary path of $C$ contradicts the $C(9)$ condition, since it can be written as  the concatenation of at most $8$ pieces: one corresponding to each intersection of $C$ with an essential corner between $C$ and $C'$ or $C''$, and one corresponding to each intersection	$\partial C \cap  \partial D, \partial C \cap \partial E, \partial C \cap \partial C'$, and $\partial C \cap \partial C''$ where $C'$ and $C''$ are the pseudo-grids or cone-cells adjacent to $C$ on $A_\beta$. This proves the base case of the claim in case~\eqref{clmit:1}.

We now handle case~\eqref{clmit:2}. By hypothesis,  there is a vertex   $v_{j_i} \in \{v_{j_1}, \ldots, v_{j_m}\}$ that does not lie in $\sigma$, and a corresponding cone-cell  or hyperplane carrier $C$ of $A_\beta$ arising from the intersection $A_\beta \cap   X_{v_{j_i}}$.  Consider the longest chain of cells of $A_\beta$ coming from elements in $\sigma$ and let  $j$ be the first index not corresponding to such a $X_{v_{j_i}}$.
  If $\partial C > X_{v_{k_0}}$, then since $S$ is adjacent to $X_{v_{j}}$  and $X_{v_{j}}< X_{v_{k_0}}$, then $\sigma$ must contain a vertex  $v_n$ with $X_{v_n} < X_{v_{j}}$. 
    Also by hypothesis, $X_{v_{j}}< \ldots < X_{v_{j_m}} < X_{v_{k_0}}$, so there exists a simplex $\theta$ satisfying that $\theta \cup \{v_\ell\}$ is a simplex for each $\ell \in \{n, j_i, \ldots, j\}$. Note that  such an $\ell$ exists since $\{n, j_i, \ldots, j\}$ contains $n$.  Let $\theta^+:=\theta \cup \{v_n, v_{j_i}, \ldots, v_{j_n}\}$.
  As before, the Helly Property in Lemma~\ref{prop:helly} and its strengthened form in Lemma~\ref{prop:stronghelly} imply together that $\bigcap_{v \in \sigma \cup \theta^+}X_v$ is simply connected, 
  so   an essential path $\gamma \rightarrow A_\beta$  bounds a disc diagram $E$ in $\widetilde X^*$ whose cubical part factors through $\bigcup_{v \in \sigma \cup \theta^+}X_v$, 
  and arguing as in the previous paragraph, either $\partial S=\partial E$, implying that $\partial S < X_{v_{k_0}}$ or the $C(9)$ small-cancellation condition is again violated. Thus, the base case of the induction is complete.

Assuming the claim when $\area (D) < N$, let $D$ be a disc diagram as hypothesised and with $\area (D) =N$. 
By Greendlinger's Lemma, $D$ is either a single cone-cell, a ladder, or has at least $3$ shells, cornsquares, or spurs. 
As usual, we may assume that $D$ has no spurs, since these can be removed without changing the interior of $D$. 
Since the case $N=1$ has been handled in the base case, we may assume $N \geq 2$, and so $D$ is a ladder or has at least 3 shells or cornsquares.
Let $S_1, S_2$ be either shells or cornsquares on $\partial D$. If both $S_1, S_2$ are shells, then the argument can be completed as follows.
Consider the diagram $D_{\beta'}$ obtained from $D_\beta$ by removing  $S_1$, so $D_{\beta'}$ is collared by the annular diagram $A_{\beta'}$ obtained from  $A_\beta$ by removing the cells in the outerpath of $\partial S$ and adding $S_1$.
Every shell of $D_\beta - S_1$ is still a shell of  $D_{\beta'}$, so by the induction hypothesis, for each shell $S_i$ with $1<i\leq n$ there is a $X_{v_{i}}$ with  $X_{v_{i}}< X_{v_{M'}}$ and $\partial S_i \rightarrow X_{v_{i}}$, where $X_{v_{M'}}$ is the maximal element in the ordering corresponding to a cone cell or hyperplane carrier in $A_{\beta'}$. 

If $\partial S_1 \rightarrow X_{v_{M'}}$ this does not lead to a contradiction, but in that case,  
repeating the same construction as above to obtain a disc diagram $D_{\beta''}$  collared by an annular diagram $A_{\beta''}$, but this time removing a shell $S_j\neq S_1$, the induction hypothesis again implies that  $\partial S_i \rightarrow X_{v_{i}}$ with $X_{v_{i}}< X_{v_{M''}}$ 
for each shell $S_i$ with $1\leq i\leq n$ and $i\neq j$ and  where $X_{v_{M''}}$ is defined analogously to $X_{v_{M'}}$, i.e., it is the maximal element in the ordering corresponding  to a cone cell or hyperplane carrier in $A_{\beta''}$.
If $S_j \rightarrow X_{v_{M''}}$, then  $\partial S_1 \rightarrow  X_{v_{1}}$ and $ \partial S_j \rightarrow X_{v_{j}}$ where $X_{v_{1}}$ and $X_{v_{j}}$ satisfy $  X_{v_{j}}< X_{v_{1}}$ and $ X_{v_{j}}> X_{v_1}$, which is impossible. 
Thus, either $ \partial S_1 \nrightarrow X_{v_{M'}}$ or $ \partial S_j \nrightarrow X_{v_{M''}}$, and since $X_{v_M}$ intersects  either  $A_{\beta'}$ or $A_{\beta''}$ non-trivially -- as it can only be excluded when removing one of the shells --, then $\partial S_i \rightarrow X_{v_i}$ where  $X_{v_i}< X_{v_{k_0}}$ for each shell $S_i$ with $1 \leq i\leq n$.

If at least one of $S_1, S_2$, say $S_1$, is a cornsquare instead of a shell, then the argument can be modified as follows. 
After shuffling the diagram if needed, we may assume that the cornsquare is an actual corner. Note that the outerpath $\mathcal{O}$ of $S_1$ has to intersect either $2$ distinct cone-cells, or a cone-cell and a pseudogrid, since if both edges of $\mathcal{O}$ are contained in a single pseudogrid or cone-cell, then the corner can be pushed out of $D_\beta$, contradicting the minimal complexity of $D_\beta$.
Let $C,C'$ denote the cone-cells or pseudogrids of $A_\beta$ intersecting $\mathcal{O}$, and let  $\rho=\partial C \cap \partial C'$.
Consider the diagram $A_{\beta}^{+}$ obtained from $A_{\beta}$ by cutting along $\rho$ to obtain a disc diagram $L_\rho$, and then attaching $S_1$ to $L_\rho$ along its intersections $\rho\cap \partial C$ and $\rho \cap \partial C'$, so that the resulting diagram is again an annular diagram. 
Note that $A_{\beta}^{+}$ is still a thick annuladder, it is still reduced,  and  it collars a disc diagram $D'$ which is a subdiagram of $D$ of smaller area. 
The induction hypothesis thus implies that all shells and cornsquares in $D'$ arise from subcomplexes $X_{v_{\ell}}$ of $\mathcal{U}$ with $X_{v_{\ell}} < X_{v_{M'}}$ where $X_{v_{M'}}$ corresponds to the maximal element intersecting $A_{\beta}^{+}$ in a cone-cell or pseudogrid. 
As before, it may be that $ S_1 \subset X_{v_{M'}}$, which  does not immediately lead to a contradiction. But again we may repeat the construction using $S_2$ instead of $S_1$ -- as in the previous paragraph if $S_2$ is a shell, or as in this paragraph if $S_2$ is a cornsquare. 
We thus obtain a new annular diagram with maximal cone-cell or pseudogrid associated to some $X_{v_{M''}}$, and $ \partial S_2 \nrightarrow X_{v_{M''}}$, as in the previous paragraph.
We conclude that, in any case, $\partial S_i \rightarrow X_{v_{i}}$  where $X_{v_i}< X_{v_{k_0}}$ for each shell $S_i$ with $1 \leq i\leq n$.
\end{proof}

We are ready to prove Lemma~\ref{clm:connected induction}:

\begin{proof}[Proof of Lemma~\ref{clm:connected induction}]
The proof is by induction on $k$. 
For the base case, the result follows from Theorem~\ref{thm:combointersections}.

 Now let $x$ and $y$ be vertices lying in distinct components of  the intersection $\bigcup_{j < k_0} X_{v_j} \cap X_{v_{k_0}}$. 
Let $X_v, X_{v'}$ be the corresponding elements of $\mathcal{U}$ containing $x$ and $y$. 
Note that $x$ and $y$ are connected by a path $\tau$ in $\bigcup_{j \leq k} X_{v_j}$, since by the induction hypothesis, this union is path-connected. 
Let $I_\tau$ be such that $\bigcup_{I_\tau} X_{v_j}$ are the elements of $\mathcal{U}$ traversed by $\tau$, and let $\bigcirc_\tau$ be either a thick annuladder or thick Möbiusladder  in $\bigcup_{I_\tau} X_{v_j}\cup X_{v_{k_0}}$ such that $\tau$ factors through $\bigcirc_\tau$.
So  the boundary path(s) of $\bigcirc_\tau$ are cycles in $\bigcup_{I_\tau} X_{v_j}\cup X_{v_{k_0}}$, and $\bigcirc_\tau$ collars a reduced disc diagram $D_\tau$ in $\widetilde X^*$.  Amongst all possible paths satisfying the above, choose $\tau$ so that $|I_\tau|$ is the least possible, and choose $D_\tau$ to have the least number of cells amongst all disc diagrams collared by $\bigcirc_\tau$.

If $\area (D_\tau)=0$, then $D_\tau$ is  a tree. Note that if $D_\tau$ has branching, then removing an edge (or several edges corresponding to a piece) from $D_\tau$ corresponds to shortening $\tau$ by pushing it away from a pair $X_{v_j}, X_{v_{j'}}$ and towards $X_{v_{k_0}}$; such a reduction contradicts the choices in the previous paragraph, so we may assume that $D_\tau$ is a possibly degenerate subpath of $X_{v_{k_0}}$ with $\partial D_\tau= \tau\tau^{-1}$. Since $\tau \subset \bigcup_{j < k_0} X_{v_j} \cap X_{v_{k_0}}$, this contradicts the hypothesis that $x, y $ lie in  distinct connected components of $\bigcup_{j < k_0} X_{v_j} \cap X_{v_{k_0}}$.

Hence, $\area (D_\tau)\geq 1$. We may assume, by performing the same reductions to $\tau$ as in the $\area (D_\tau)= 0$ case, that $\partial D_\tau$ has no spurs, so by Greendlinger's Lemma $D_\tau$ must have at least one shell or cornsquare. Claim~\ref{clm:final} now implies that $\partial S < X_{v_{k_0}}$  for every shell or cornsquare $S$ in $D_\tau$. 
Let $S$ be a shell or cornsquare in $D_\tau$. If $S$ is a shell, it intersects at least $3$ consecutive cells $C_1,C_2,C_3$ of $\bigcirc_\tau$, and so the path $\tau'$ obtained from $\tau$ by pushing across $C_2$ traverses at most as many cells as $\tau$, but bounds a disc diagram $D_{\tau'}$ with $\area (D_{\tau'})< \area (D_\tau)$, contradicting our initial choices.  If $S$ is a cornsquare, it intersects at least $2$ consecutive cone-cells or a cone-cell and a hyperplane carriers $C_1,C_2$ of $\bigcirc_\tau$. As in the proof of Claim~\ref{clm:final}, let $\rho = C_1 \cap C_2$, and consider the diagram $\bigcirc_{\tau'}$ obtained from $\bigcirc_\tau$ by cutting along $\rho$ and attaching $S$ to $C_1$ and  $C_2$  along the corresponding intersections. The  diagram $\bigcirc_{\tau'}$ collars a disc diagram $D_{\tau'}$ with $\area (D_{\tau'})< \area (D_\tau)$, so we may again apply Claim~\ref{clm:final}.

Thus, we arrive at a contradiction in all cases, so  $\bigcup_{j < k_0} X_{v_j} \cap X_{v_{k_0}}$ must be connected and the induction is complete. 
\end{proof}

We use Lemma~\ref{clm:connected induction} to prove the following:

\begin{lemma}\label{clm:pi1 induction} Let $ X^*$ be  a $C(9)$ cubical presentation, let $V=\{v_j\}_{j \in \naturals}$ be the vertices of its structure graph $\Lambda$, and for each $v_j \in V$, let $X_{v_j}$ be the corresponding element of $\mathcal{U}$. For each $k \in \naturals$, there is an isomorphism 
$$\pi_1(\bigcup_{j \leq k} X_{v_j}) \cong \ast_{j \leq k} \pi_1X_{v_j}.$$
\end{lemma}

\begin{proof}
The proof is by induction on $k \in \naturals$. The base case follows from Theorem~\ref{thm:combointersections}. For the inductive step,
since  $\bigcup_{j < k_0} X_{v_j} \cap X_{v_{k_0}}$ is connected by the previous lemma, we may apply the Seifert Van-Kampen Theorem to $\bigcup_{j < k_0} X_{v_j} \cup X_{v_{k_0}}$. This yields the desired isomorphism provided that either $\bigcup_{j < k_0} X_{v_j} \cap X_{v_{k_0}}$ is simply-connected or both inclusions $$\bigcup_{j < k_0} X_{v_j} \hookleftarrow\bigcup_{j < k_0} X_{v_j} \cap X_{v_{k_0}}\hookrightarrow  X_{v_{k_0}}$$ have trivial image.

If $\bigcup_{j < k_0} X_{v_j} \cap X_{v_{k_0}}$  is not simply-connected, then there is an essential closed path $\sigma \rightarrow \bigcup_{j < k_0} X_{v_j} \cap X_{v_{k_0}}$, we may moreover assume that $\sigma$ is a shortest path in its homotopy class.  The map $\sigma \rightarrow \bigcup_{j < k_0} X_{v_j} \cap X_{v_{k_0}}$ factors through $\bigcup_{j < k_0} X_{v_j}$ and $X_{v_{k_0}}$ via the respective inclusions. Abusing notation, let $\sigma \rightarrow \bigcup_{j < k_0} X_{v_j}$ and $\sigma \rightarrow  X_{v_{k_0}}$ be the corresponding paths.
If $\sigma \rightarrow \bigcup_{j < k_0} X_{v_j}$ is  essential, then  the image of $\sigma$ in $\bigcup_{j < k_0} X_{v_j}$ is covered by a finite number of $X_{v_j}$'s with $j < k_0$, and there is an annular or Möbius diagram $\bigcirc_\sigma \rightarrow \bigcup_{j < k_0} X_{v_j}$ such that $\sigma \rightarrow \bigcup_{j < k_0} X_{v_j}$ factors through $\bigcirc_\sigma$. 
Amongst all such $\bigcirc_\sigma$'s, we may choose one with minimal complexity. Let $\{X_{v_{j_0}}, \ldots, X_{v_{j_n}}\}$ with $j_i \leq j$ be the elements of $\mathcal{U}$ intersecting $\bigcirc_\sigma$.
Applying  the Seifert Van-Kampen Theorem to $\bigcirc_\sigma$, we see that $\bigcirc_\sigma$ cannot be covered by two contractible sets with connected intersection, so in particular, viewing $\sigma$ as a concatenation of two arcs $\tau'$ and $\tau''$ where $\tau'$ traverses a single $X_{v_{j_0}}$ with $1\leq j_0 < k_0$ and $\tau''$ traverses $\bigcup_{j < k_0, j \neq j_0} X_{v_j}$, it follows that the intersection $\bigcup_{j < k_0, j \neq j_0} X_{v_j} \cap X_{v_{j_0}}$ must be disconnected. This contradicts Lemma~\ref{clm:connected induction}.

Thus, $\sigma \rightarrow \bigcup_{j < k_0} X_{v_j}$ must be nullhomotopic. We now argue that $\sigma \rightarrow X_{v_{k_0}}$ is also nullhomotopic. Indeed, if this isn't the case, then   $\sigma \rightarrow \bigcup_{j < k_0} X_{v_j}\cup X_{v_{k_0}}$ is nullhomotopic by the Seifert Van-Kampen Theorem, but this contradicts the fact that the map $X_{v_{k_0}} \rightarrow \hat X$ is a local isometry (since $X_{v_{k_0}}$ is locally convex), and thus $\pi_1$-injective. 

Since this is the case for each closed essential path $\sigma \rightarrow \bigcup_{j < k_0} X_{v_j} \cap X_{v_{k_0}}$, we conclude that both of the inclusions
$$\bigcup_{j < k_0} X_{v_j} \hookleftarrow\bigcup_{j < k_0} X_{v_j} \cap X_{v_{k_0}}\hookrightarrow  X_{v_{k_0}}$$ have trivial image, and the result follows.
\end{proof}

Using Lemma~\ref{clm:pi1 induction}, we can now prove that the Cohen-Lyndon property holds in our setting.

\begin{theorem}\label{thm:clprop}
Let $X^*=\langle X \mid \{Y_i\}_{i \in I} \rangle$ be a cubical presentation. If $X^*$ satisfies the $C(9)$ condition, then ($\pi_1 X, \{\pi_1 Y_i\}_{i \in I}$) has the Cohen-Lyndon property.
\end{theorem}

\begin{proof}
 Recall that $\hat X$ is the covering space of $X$ corresponding to the subgroup $ker(\pi_1 X \rightarrow \pi_1 X^*)<\pi_1X$. 

Fix a basepoint $\boldsymbol{\cdot} \in \hat X$. From Lemma~\ref{clm:pi1 induction}, there is, for each $k \in \naturals$, an isomorphism $$\pi_1(\bigcup_{j \leq k} X_{v_j}, \boldsymbol{\cdot}) \cong \ast_{j \leq k} \pi_1(X_{v_j},\boldsymbol{\cdot}).$$

Now $\hat X=\bigcup_{j \leq k, k \rightarrow \infty} X_{v_j}$ by Lemma~\ref{lem:iscover}, and  $\ast_{j \leq k} \pi_1X_{v_j}=\pi_1 \bigvee_{j \leq k} X_{v_j}$ for each $k \in \naturals$. Since each $X_{v_j}$ is aspherical, $\bigvee_{j \leq k} X_{v_j}$ is aspherical, as is $\hat X$, by virtue of being a non-positively curved cube complex. 
Thus $\hat X$ and $\bigvee_{j \leq k, k\rightarrow \infty} X_{v_j}$ are homotopy equivalent, since classifying spaces are unique up to homotopy equivalence.
Moreover, by choosing for each $k \in \naturals$ a path $\rho_k \rightarrow \bigcup_{j \leq k} X_{v_j}$ joining the basepoint  to some $p_k \in \bigcup_{j \leq k} X_{v_j}$ and attaching each $\rho_k$ to the corresponding $X_k$, the wedge  $\bigvee_{j \leq k, k\rightarrow \infty} X_{v_j}$ can be realised as a subcomplex of $\hat X$, and the inclusion $\bigvee_{j \leq k, k\rightarrow \infty} X_{v_j} \hookrightarrow \hat X$ is a homotopy equivalence, so there is a deformation retraction
$$\hat X \rightarrow \bigvee_{j \leq k, k\rightarrow \infty} X_{v_j}.$$
This implies that the isomorphism $\pi_1(\bigvee_{j \leq k, k\rightarrow \infty} X_{v_j}, \boldsymbol{\cdot}) \cong  \pi_1(\hat{X},\boldsymbol{\cdot})$ is the canonical isomorphism induced by the inclusion, and thus that $\pi_1(\bigvee_{j \leq k, k\rightarrow \infty} X_{v_j}, \boldsymbol{\cdot}) =  \pi_1(\hat{X},\boldsymbol{\cdot})$, since $\pi_1(\bigvee_{j \leq k, k\rightarrow \infty} X_{v_j}, \boldsymbol{\cdot})$ is a subgroup of $\pi_1(\hat X,\boldsymbol{\cdot})$ that generates it.
Now, for each $j \in \naturals$ the complex $X_{v_j}$ is either contractible (if it is a supporting hyperplane carrier or an untethered component) or is equal to a translate $gY_i$ for some $i \in I$. So, dropping the basepoint:  
$$\nclose{\{\pi_1 Y_i\}}=\pi_1\hat X = \pi_1\big (\bigvee_{gY_i \in \pi_1X/Stab_{\pi_1X}(\widetilde Y_i)} gY_i\big )= \ast_{i\in I, g \in T_i} \langle \pi_1 Y_i\rangle ^g $$ 
where $T_i$ are left transversals of $Stab_{\pi_1X}(\widetilde Y_i)\nclose{\pi_1 Y_i}$ in $\pi_1X$.
\end{proof}

\section{Asphericity}\label{sec:mains}

\subsection{Asphericity of the reduced space}

We recall the following construction, introduced in~\cite{Arenas2023pi2}.

Let $X^*= \langle X \mid \{Y_i\}_{i \in I} \rangle$ be a cubical presentation and consider the universal cover $\widetilde{X^*}$ of the coned-off space. 
Note that $X$ is a subspace of $X^*$, so the preimage $\hat X$ of $X$ in $\widetilde{X^*}$ is a covering space of $X$, namely the regular cover corresponding to $ker(\pi_1X \rightarrow \pi_1X^*)$. Consider the universal cover $\widetilde X $ of $X$.
For each $i \in I$, and for any fixed elevation $\widetilde Y_i \rightarrow \widetilde X$ of $Y_i$,  we have that $\pi_1 (Y_i,y_{i_0}) < Stab_{\pi_1(X,x_0)}(\widetilde Y_i)$. %If $X^*= \langle X |\{Y_i\}_{i \in I} \rangle$ is minimal, then by definition $\pi_1 (Y_i,y_{i_0}) = Stab_{\pi_1(X,x_0)}(\widetilde Y_i)$, but in general this is not the case.

Let $\{g_\ell\pi_1 Y_i\}$ be coset representatives of  $\pi_1 (Y_i,y_{i_0})$ in $Stab_{\pi_1(X,x_0)}(\widetilde Y_i)$. 
The elevations $\{g_\ell Y_i \rightarrow \hat X\}$ have the same image in $\hat{X}\subset \widetilde{X^*}$, so their cones are all isomorphic in $ \widetilde{X^*}$. 
Thus, there is a quotient $\cup_{\ell}g_\ell Cone(Im_{\hat X} (Y_i)) \rightarrow Cone(Im_{\hat X} (Y_i))$ where all cones over $Im_{\hat X}(Y_i)$ are identified to a single cone. This extends to a quotient $\widetilde{X^*} \rightarrow \bar X^*$, which we call the \emph{reduced space} of $\widetilde{X^*}$.

\begin{definition}
A cubical presentation $X^*=\langle X \mid \{Y_i\}_I \rangle$ is \emph{reduced} if $X^*=\bar X^*$.
\end{definition}

When all the $Y_i$'s in the cubical presentation are compact, reducibility of $X^*$ is equivalent to minimality -- a cubical presentation is \emph{minimal} if $Stab_{\pi_1X}(\widetilde{Y_i})=\pi_1 Y_i$ for each $i \in I$ and for each elevation $\widetilde{Y_i} \rightarrow Y_i \rightarrow X$  of $Y_i \rightarrow X$ (really we have to fix a basepoint $x \in X$ for this equality to make sense, but we suppress it from the notation).

Minimality generalises prohibiting relators that are proper powers in the classical small-cancellation case. In our setting, minimality is used to rule-out the torsion that could be created in a quotient $\pi_1X/\langle \langle \{\pi_1 Y'_i\} \rangle \rangle$ if the  $Y'_i \rightarrow X$  are themselves non-trivial finite covers $Y'_i \rightarrow Y_i \rightarrow X$. The simplest example of a non-minimal cubical presentation is given by  $X^*=\langle X \mid Y \rangle$ where $Y$ is a finite degree covering of $X$. Such an  $X^*$ will have finite fundamental group even if it satisfies the $C(9)$ condition.

The starting point in our strategy is Theorem~\ref{thm:cub2}, which is~\cite[Thm 3.5]{Arenas2023pi2}. 

\begin{theorem}\label{thm:cub2} Let $X^*= \langle X \mid \{Y_i\} \rangle$ be a cubical presentation that satisfies the $C(9)$ condition. Let $\bar X^*$ be the reduced space of $\widetilde{X^*}$. 
Then $\pi_2\bar X^*=0$.
\end{theorem}

We are ready to prove Theorem~\ref{thm:mainintro} from the introduction.

\begin{theorem}\label{thm:asph}
Let $X^*=\langle X \mid \{Y_i\} \rangle$ be a  cubical presentation. If $X^*$ satisfies the $C(9)$ condition, then $\bar{X}^*$ is aspherical.
In particular, if $X^*$ is reduced, then it is a model for $K(\pi_1X^*,1)$. 
\end{theorem}

\begin{remark} The ``$dim(X)\leq 2$ plus $cd(\pi_1Y_i)\leq 1$'' version of Theorem~\ref{thm:asph} is proven in~\cite[3.7]{Arenas2023pi2}. We note however that there is an inaccuracy in the statement of the theorem as originally written:  rather than reducedness, which is the property that is actually needed in the proof,  \cite[3.7]{Arenas2023pi2} assumes minimality of $X^*$. 
\end{remark}

\begin{proof}
We will show by induction that $H_i(\bar X^*)=0$ for all $i \in \naturals$. 
By Theorem~\ref{thm:cub2}, we have that $\pi_2\bar X^*=H_2(\bar X^*)=0$; now assume that $H_{k'}(\bar X^*)=0$ for all $k'< k$.
Fix a basepoint $x_0 \in X$; for ease of notation, set 
$$\bigsqcup_i \bigsqcup_{gStab_{\pi_1(X,x_0)}(\widetilde Y_i) \in \pi_1X/Stab_{\pi_1(X,x_0)}(\widetilde Y_i)}  gY_i:= \mathbf{Y} $$
and 
$$ \bigsqcup_i \bigsqcup_{gStab_{\pi_1(X,x_0)}(\widetilde Y_i) \in \pi_1X/Stab_{\pi_1(X,x_0)}(\widetilde Y_i)}  gCone(Y_i):= \mathbf{Cone(Y)}.$$ 
Since $\hat X \cap \mathbf{Cone(Y)}= \mathbf{Y}$, we have the Mayer-Vietoris sequence:

\[
    \begin{tikzcd}[arrows=to]
        \cdots \rar & H_k(\mathbf{Y}) \rar & H_k(\hat X) \oplus H_k (\mathbf{Cone(Y)}) \rar & H_k(\bar X^*) \rar & \hphantom{0}\\
        \hphantom{\cdots} \rar 
        & H_{k-1}(\mathbf{Y}) \rar 
        & \makebox[\widthof{$H_k(A) \oplus H_k(B)$}][c]{$\cdots\hfill \cdots$} \rar
        &  H_0(\hat X^*) \rar & 0
    \end{tikzcd}
\]

 Since $H_{k-1}(\bar X^*)=0=H_{k-1} (\mathbf{Cone(Y)})$, we see that the map $H_{k-1}(\mathbf{Y}) \longrightarrow H_{k-1}(\hat X)$ is surjective. 
 We claim that, in fact, $H_{k-1}(\mathbf{Y}) \longrightarrow H_{k-1}(\hat X)$ is an isomorphism and therefore the map $H_k(\bar X^*)  \longrightarrow H_{k-1}(\mathbf{Y})$ has trivial image, and so $H_k(\bar X^*)=0$. Indeed, Theorem~\ref{thm:clprop} implies that $\pi_1\hat X=ker(\pi_1X \rightarrow \pi_1X^*) \cong \ast_{g \in \pi_1X^*, i \in I} \pi_1g Y_i$, and that $\hat X \sim \bigvee_i Y_i$. Thus $H_{k-1}\hat{X}=H_{k-1} \bigvee_i Y_i=H_{k-1}(\mathbf{Y})$.

Applying Hurewicz's Theorem iteratively to $\hat X^*$, starting with $H_2(\bar X^*)\cong \pi_2\bar X^*=0$, we see that  $H_k(\bar X^*)\cong \pi_k\bar X^*=0$ for each $k \in \naturals$.
\end{proof}

It is perhaps not immediately clear to what extent the ``reduced'' hypothesis in the second part of Theorem~\ref{thm:asph} is natural or reasonably achievable, specially since, in many situations when trying to produce a cubical presentation for a quotient $G/\nclose{\{H_i\}}$, it is easier to show that a cubical presentation $Z^*$ whose relators have large stabilisers satisfies a small-cancellation condition rather than showing that another -- perhaps more natural -- cubical presentation $X^*$ whose relators have small stabilisers does. One way to circumvent this problem is as follows. If the cubical parts $\hat X$ and $\hat Z$ of $X^*=\langle X\mid \{Y_i\}\rangle$ and $Z^*=\langle Z\mid \{W_i\}\rangle$ are cubically isomorphic,  $\bar X^*$ is aspherical, and  $\bar X^*$ can be obtained from  $\widetilde Z^*$ by identifying simplices $s ,s'$ whenever they arise from coning over the same $n$-cube in $\hat X$, then $\widetilde Z^*$ is also aspherical. This gives a way of applying  Theorem~\ref{thm:asph} to $Z^*$ to deduce asphericity of $X^*$.

We can generalise this observation, and use the asphericity of $\bar{X}^*$ to build classifying spaces in much greater generality:

\begin{corollary}\label{cor:asphericalreloaded}
Let $X^*=\langle X \mid \{Y_i\}_{i \in I} \rangle$ be a  cubical presentation that satisfies the $C(9)$ condition, and for each $i \in I$, let $\mathcal{Y}_i$ be a model for $EStab_{\pi_1X^*}(Y_i)$ for which there is an embedding $\phi_i:Y_i \hookrightarrow \mathcal{Y}_i$. Let
$$\mathcal{X}=\hat{X} \bigcup_{\{gY_i\}} g\mathcal{Y}_i$$ 
where each  $g\mathcal{Y}_i$ is glued to $\hat X$ along the corresponding  elevation  $gY_i$ of $Y_i\rightarrow X$.
Then $\mathcal{X}$ is a model for $E\pi_1X^*$.
\end{corollary}

\begin{remark}
The space $\mathcal{X}$ in Corollary~\ref{cor:asphericalreloaded} is a kind of ``bordification'' of $\bar X^*$, in the spirit of constructions such as those in~\cite{Siegel59,BS73} and~\cite{BestvinaFeighn00}. 
\end{remark}

\begin{proof}
It follows from the Seifert Van-Kampen Theorem that  $\mathcal{X}$ is simply connected. 
Since $E\pi_1X^*$ acts freely on $\widetilde X^*$, the action on  $\hat X$, which originally permuted the cones, induces an action on $\mathcal{X}$ which permutes the $\mathcal{Y}_i$'s over different $Y_i$'s and extends to the action of $Stab_{\pi_1X^*}(Y_i)$ on the corresponding $\mathcal{Y}_i$. 
Thus  $E\pi_1X^*$ acts freely on  $\mathcal{X}$. % is a $\pi_1X$-CW complex. 
There is a Mayer-Vietoris sequence:
\[
    \begin{tikzcd}[arrows=to]
        \cdots \rar & H_k(\cup_i Y_i) \rar & H_k(\hat X) \oplus H_k (\cup_i \mathcal{Y}_i) \rar & H_k(\mathcal{X}) \rar & \hphantom{0}\\
        \hphantom{\cdots} \rar 
        & H_{k-1}(\cup_i Y_i) \rar 
        & \makebox[\widthof{$H_k(A) \oplus H_k(B)$}][c]{$\cdots\hfill \cdots$} \rar
        &  H_0(\mathcal{X}) \rar & 0
    \end{tikzcd}
\]

Since we have the homotopy equivalence $X\sim\bigvee_{i \in I} Y_i$ by Theorem~\ref{thm:clprop}, then $H_{k}(\hat{X})=H_{k}(\mathbf{Y})$ for each $k \in \naturals$, and  $H_k (\mathcal{Y}_i)=0$ since $\mathcal{Y}_i$ is a model for $EStab(Y_i)$. For each $k \in \naturals$, we get short exact sequences 
\[
    \begin{tikzcd}[arrows=to]
        0\rar & H_k(\mathbf{Y}) \rar & H_k(\hat X) \rar & H_k(\mathcal{X}) \rar & 0
    \end{tikzcd}
\]
 and so $H_k(\mathcal{X})=0$ for every $k$.   We now argue that $\pi_2 \mathcal{X}=0$.  Indeed, the proof of Theorem~\ref{thm:cub2} in~\cite{Arenas2023pi2} uses the $C(9)$ condition and the fact that cones are contractible, but not the cellular structure of the cones. Another way to see this is as follows. 
 The complex $\mathcal{X}$ is homotopy equivalent to $\hat X\cup \{gM_{\phi_i}\}$ where each $M_{\phi_i}$ is the mapping cylinder of  $\phi_i:Y_i \rightarrow \mathcal{Y}_i$. Since each  $\mathcal{Y}_i$ is contractible, collapsing each $\mathcal{Y}_i \subset X\cup \{gM_{\phi_i}\}$ to  a point is a homotopy equivalence, the result of which is $\bar X^*$. Now $\pi_2 \mathcal{X}=\pi_2 \bar X^*=0$ by Theorem~\ref{thm:cub2}. 
 To conclude the proof, we apply  Hurewicz's Theorem  iteratively to deduce that $\pi_n \mathcal{X}=H_n(\mathcal{X})=0$, as in the proof of Theorem~\ref{thm:asph}. 
\end{proof}

Let $n \in \naturals$. A group $G$ is \emph{type} $F_n$ if it admits a model for $K(G,1)$ with compact $n$-skeleton, it is \emph{type} $F_\infty$ if it is type $F_n$ for all $n \in \naturals$, and it is \emph{type} $F$ if it admits a compact $K(G,1)$.
From~\ref{cor:asphericalreloaded} we obtain:

\begin{corollary}\label{cor:finprop}
Let $\Gamma=\pi_1X^*$ where $X^*=\langle X \mid \{Y_i\}_{i \in I} \rangle$  is a  cubical presentation that satisfies the $C(9)$ condition and each $Y_i$ is compact. If $\pi_1X$  and each $Stab_{\pi_1X^*}(Y_i)$ is  type $F_n$ (resp., $F$), then $\Gamma$ is type $F_n$ (resp., $F$).
\end{corollary}

\begin{proof}
Assume $\pi_1X$  and each $Stab_{\pi_1X^*}(Y_i)$ is  type $F_n$ (resp., $F$). Let  $\mathcal{Y}_i$  be a model for $EStab_{\pi_1X^*}(Y_i)$  having a cocompact $n$-skeleton (resp.,  $\mathcal{Y}_i$ is cocompact) and such that $Y_i$ embeds in  $\mathcal{Y}_i$ (for instance,  $\mathcal{Y}_i \times Cone(Y_i)$ will do). Now apply Corollary~\ref{cor:asphericalreloaded}.
\end{proof}

The construction of $\mathcal{X}$ in Corollary~\ref{cor:asphericalreloaded}, and the proof of the corollary, suggest  a general method for obtaining classifying spaces from  the Cohen-Lyndon property. Since it might be useful for other applications, we state it below:

\begin{theorem}\label{thm:asphericalreloaded}
Let $X$ and $\{Y_i\}_{i \in I}$ be  aspherical CW-complexes, assume each  $Y_i \hookrightarrow X$ is a $\pi_1$-injective immersion, and let $\Gamma=\pi_1X/\nclose{\pi_1Y_i}$. 
For each $i \in I$, let $\mathcal{Y}_i$ be a model for $EStab_\Gamma(Y_i)$, and assume each $Y_i \hookrightarrow \mathcal{Y}_i$ is an embedding. Let $\hat X \rightarrow X$ be the covering corresponding to the subgroup $ker(\pi_1X \rightarrow \Gamma)<\pi_1X$, assume each elevation of $Y_i$ to $\hat X$ is an embedding,  and let $\mathcal{X}$ be the quotient 
$$\mathcal{X}=\hat X \bigcup_{\{gY_i\}} g\mathcal{Y}_i$$ where each  $ \mathcal{Y}_i$ is glued to $\hat X$ along the corresponding elevation of $Y_i$.
If the pair $(\pi_1 X, \{\pi_1 Y_i\}_{i \in I})$ has the Cohen-Lyndon property and $\pi_2 \mathcal{X}=0$, then  $\mathcal{X}$
is a model for $E\Gamma$.
\end{theorem}

If $X^*$ is a reduced cubical presentation that satisfies the $C(9)$ condition, then an upper bound on the cohomological dimension of $\pi_1 X^*$ in terms of $\dim(X)$ follows immediately from Theorem~\ref{thm:asph}. Namely, since the cohomological dimension of $\pi_1 X^*$ is bounded above by its geometric dimension, which is in turn bounded above by the dimension of $X^*$, then 
$$cd(X^*)\leq\max\{dim(X),\max_i\{dim(Y_i)\}+1\},$$ 
and a similar inequality can be obtained when $X^*$ is not reduced by using the classifying space $\mathcal{X}/\pi_1X^*$ built in Corollary~\ref{cor:asphericalreloaded}. 
However, a sharper upper bound and a lower bound on $cd(\pi_1X^*)$ also follows from Theorem~\ref{thm:clprop}. As another immediate consequence, we obtain formulas for the homology and cohomology groups of $\pi_1 X^*$; we collect these results as follows:

\begin{corollary}\label{cor:dimensions}
Let $X^*=\langle  X \mid \{Y_i\}\rangle$ be a reduced $C(9)$ cubical presentation. Then
$$cd(\pi_1X)-\max_i\{cd(\pi_1Y_i)\}\leq cd(\pi_1X^*)\leq \max\{cd(\pi_1X),\max_i\{cd(\pi_1Y_i)\}+1\}.$$
Moreover, for every $\pi_1 X^*$-module $A$ and for all $n\geq \max\{cd(\pi_1Y_i)\}+2$: $$H^n(\pi_1X^*, A)=H^n(\pi_1X, A) \text{ and } H_n(\pi_1X^*, A)=H_n(\pi_1X, A).$$ 
\end{corollary}

\begin{proof}
The upper bound on the cohomological dimension follows from Theorem~\ref{thm:clprop} via~\cite[Corollary 2.2(ii)]{PetSun21}, applied to $(\pi_1X, \{\mathcal{N}(\pi_1Y_i)\}, \{\pi_1 Y_i\})$, since, as $X^*$ is reduced, then the normaliser $\mathcal{N}(\pi_1 Y_i)=\pi_1 Y_i$ for each $i$.
Using Theorem~\ref{thm:clprop}, the lower bound follows from the short exact sequence $$1 \rightarrow \nclose{\{\pi_1 Y_i\}} \rightarrow \pi_1 X \rightarrow \pi_1 X^* \rightarrow 1 $$ 
which implies the inequality $cd(\pi_1 X) \leq cd(\nclose{\{\pi_1 Y_i\}}) + cd(\pi_1 X^*)$ by,  for instance, \cite[VIII 2.4]{BrownBook82and94}.
The formulae for the homology and cohomology of $\pi_1X^*$ follow from Theorem~\ref{thm:clprop} via~\cite[Corollary 2.2(i)]{PetSun21}, again applied to the triple $(\pi_1X, \{\mathcal{N}(\pi_1Y_i)\}, \{\pi_1 Y_i\})$.
\end{proof}

\subsection{Asphericity for proper actions}
Let $G$ be a discrete a group.
 A \emph{classifying space  for proper actions}  is a $G$-CW-complex $\underbar EG$ satisfying: 
\begin{enumerate}
\item every cell stabiliser is finite,
\item  the fixed point space  $(\underbar EG)^H$ is contractible for every finite  $H < \pi_1X^*$.
\end{enumerate}

\begin{comment}
Examples of families include:\begin{enumerate}
\item $F=\mathcal{TR}$,  the family containing only the trivial group, in which case $E_FG$  coincides with the ``usual'' classifying space $EG$.
\item $F=\mathcal{FIN}$, the family of finite subgroups of $G$, in which case $E_FG$ is also known as  the \emph{classifying space for proper actions} of $G$, denoted $\underbar EG$.
\item $F=\mathcal{VCYC}$, the family of virtually cyclic subgroups of $G$, in which case $E_FG$ is also denoted $\uuline{E}G$.
\item $F=\mathcal{VAB}$, the family of virtually abelian subgroups of $G$.
\end{enumerate}
\end{comment}

 Theorem~\ref{thm:asph} provides classifying spaces for  the fundamental groups of reduced $C(9)$ cubical presentations; similar results providing classifying spaces for proper actions can be derived  when the cubical presentation satisfies the $C(9)$ condition, together with  a suitable hypothesis that generalises being reduced.

We will prove the following theorem:

\begin{theorem}\label{thm:asphfinite}
Let $X^*=\langle X \mid \{Y_i\} \rangle$ be a cubical presentation. If $X^*$ satisfies the $C(9)$ condition and $[Stab_{\pi_1X}(\widetilde Y_i):\pi_1Y_i]< \infty$ for each  $Y_i \rightarrow X$, then $\bar X^*$ is a classifying space for proper actions for $\pi_1X^*$. 
\end{theorem}

We first check that the isotropy groups of the action of $\pi_1X^*$ on $\bar X^*$ are finite. 

\begin{lemma}\label{lem:finitestab}
Let $X^*=\langle X \mid \{Y_i\} \rangle$ be a  cubical presentation.   If $X^*$ satisfies the $C(9)$ condition and $[Stab_{\pi_1X}(\widetilde Y_i):\pi_1Y_i]< \infty$ for each  $Y_i \rightarrow X$,
then the action of $\pi_1X^*$ on $\bar X^*$ has finite cell-stabilisers. 
\end{lemma}

\begin{proof}
All stabilisers of cells in $\hat X \subset \bar X^*$ are trivial. By hypothesis,   $Stab_{\pi_1X}(\widetilde Y_i)/\pi_1Y_i $ is finite for each $i \in I$. Now  by \cite[3.9]{Arenas2023pi2}, $Stab_{\pi_1X^*}(v)=Stab_{\pi_1X}(\widetilde Y_i)/\pi_1{Y_i}$ where $v$ is a cone-vertex of $\bar X^*$ corresponding to a cone over some $ Y_i$. Thus, the stabilisers of cone-vertices are finite.
\end{proof}

Using results from~\cite{Arenas2023pi2}, we can now establish that $(\bar X^*)^H$ is contractible for each finite  $H <\pi_1X^*$. 

\begin{lemma}\label{lem:contfix}
Let $X^*=\langle X \mid \{Y_i\} \rangle$ be a symmetric cubical presentation that satisfies the $C(9)$ condition. If there is a uniform upper-bound on the size of pieces in any essential path $\sigma \rightarrow gY_i$, then every finite subgroup $H <\pi_1X^*$ fixes a unique point in $\bar X^*$. In particular, the fixed-point space $(\bar X^*)^H$ is contractible.
\end{lemma}

\begin{remark}
The hypothesis on the size of pieces in Lemma~\ref{lem:contfix} is not particularly restrictive: in particular, it is automatically satisfied for any  cubical presentation $X^*=\langle X \mid \{Y_i\}_{i \in I} \rangle$ having $|I|< \infty$ and satisfying the $C(2)$ condition~\cite[3.12]{Arenas2023pi2}.
\end{remark}

\begin{proof}
This is shown in~\cite[3.13]{Arenas2023pi2} and~\cite[3.14]{Arenas2023pi2}. The statement of~\cite[3.13]{Arenas2023pi2} hypothesises compactness of the $Y_i$'s, but the proof only uses the existence of a uniform upper bound on the size of pieces arising in  essential paths in the $gY_i$'s.
\end{proof}

This finishes the proof of Theorem~\ref{thm:asphfinite}.

\subsection{Virtual torsion-freeness}

This short subsection gives a sufficient condition for $\pi_1X^*$ of a  $C(9)$ cubical presentation to be virtually torsion-free. 
A version of this result, assuming  $dim(X)\leq 2$ and $cd(\pi_1Y_i)\leq 1$ for all $Y_i$, was given in~\cite[3.18]{Arenas2023pi2}. We note that the alluded result assumes that the cubical presentation is symmetric in the sense of~\cite[8.12]{WiseIsraelHierarchy}, but that this hypothesis is not actually used in the proof therein. However, we do use a form of symmetry (that $\pi_1Y_i$ has finite index in $Stab_{\pi_1X}(\widetilde Y_i)$)  to prove that $vcd(\pi_1X^*)$ is finite in Lemma~\ref{lem:vcd}.

\begin{lemma}\label{lem:vcd} Let $X^*=\langle X \mid \{Y_i\}_{i \in I} \rangle$ be a  $C(9)$ cubical presentation. If $[Stab_{\pi_1X}(\widetilde Y_i):\pi_1Y_i]< \infty$ for each  $Y_i \rightarrow X$ and there exists a finite regular cover $\check{X} \rightarrow X$  where  $Y_i \rightarrow X$ lifts to an embedding for each $i \in I$, then $\pi_1X^*$ is virtually torsion-free and $$vcd(\pi_1X^*)\leq \max\{cd(X), \max{cd(\pi_1 Y_i)+1}\}.$$
\end{lemma}

\begin{proof}
The finite-degree covering map $\check{X} \rightarrow X$ induces a finite-degree covering map $\check{X}^* \rightarrow X^*$. By~\cite[3.13]{Arenas2023pi2}, if an element $g\in \pi_1X^*-\{1\}$ satisfies $g^n=1$, then it is represented by a path $\sigma \rightarrow X^*$ such that  $\sigma^n$ is conjugate into $\pi_i Y_i$ for some $i \in I$. Since $Y_i \rightarrow \check X$ is an embedding, then $\sigma \rightarrow \check X \subset \check{X}^*$ cannot be closed, so $g \notin \pi_1\check{X}^*$. Thus $\pi_1\check{X}^*$ is torsion free.

Since $X^*$ satisfies the $C(9)$ condition, then $\check{X}^*$ also satisfies  $C(9)$, and the cones in $\check{X}^*$ are elevations of the cones in $\check{X}^*$. Now Theorem~\ref{thm:clprop} implies together with~\cite[Corollary 2.2(ii)]{PetSun21} and the hypothesis on the $Stab_{\pi_1X}(\widetilde Y_i)$'s that 
$$cd(\pi_1 \check{X}^*)\leq \max\{cd(\check{X}), \max{cd(\pi_1 \check{Y}_i)+1}\}=\max\{cd(X), \max{cd(\pi_1 Y_i)+1}\},$$
as asserted in the lemma.
\end{proof}

The hypothesis on the existence of the finite regular cover $\check{X} \rightarrow X$ in Lemma~\ref{lem:vcd} is satisfied, for instance, when $X$ is virtually special. Virtual specialness is in turn satisfied if, for example, $\pi_1 X$ is hyperbolic~\cite{AgolGrovesManning2012} or hyperbolic relative to compatible  virtually special subgroups~\cite{Reyes23,GrovesManning22}.

\section{Some applications}\label{sec:applications}

The results in Section~\ref{sec:CL prop} and Section~\ref{sec:mains}  can be applied to a large variety of quotients of cubulated groups; below we describe some salient examples.

\subsection{Cubically presenting Artin groups} \label{subsecartin}

\begin{definition}
Let $\Gamma$ be a simplicial graph whose edges $(a_i,a_j)$ are labelled with integers $m_{ij}\geq 2$. Let $(a_i,a_j)^{m_{ij}}$ denote the first half of the word $(a_ia_j)^{m_{ij}}$. The \emph{Artin group} $A_\Gamma$ is given by the presentation:
$$A_\Gamma=\langle a_1, \ldots, a_n \mid (a_i,a_j)^{m_{ij}}=(a_j,a_i)^{m_{ij}} : i\neq j \rangle.$$

Let $A_{\bar \Gamma}$ be the underlying right angled Artin group (RAAG) associated to $A_\Gamma$:
\begin{equation}\label{pres:raag}
A_{\bar \Gamma}=\langle a_1, \ldots, a_n \mid (a_i,a_j)^2=(a_j,a_i)^2 : m_{ij}=2 \text{ and } i \neq j \rangle.
\end{equation} 

Associated to $ A_{\bar \Gamma}$ is a non-positively curved cube complex $R(A_{\bar \Gamma})$ called the \emph{Salvetti complex} of $A_{\bar \Gamma}$. To construct $R(A_{\bar \Gamma})$, one starts with the presentation complex $\mathcal{X}$ associated to~\eqref{pres:raag}. This complex is obtained from a collection of $2$-tori by gluing some of them together along their edges. If $\Gamma$ contains no triangles, then $\mathcal{X}$ is a non-positively curved cube complex, and $R(A_{\bar \Gamma}):=\mathcal{X}$, otherwise, for every $n$-clique in $\Gamma$ glue in an $n$-cube along its boundary in $\mathcal{X}$. The resulting cube complex is $R(A_{\bar \Gamma})$ and is non-positively curved.  It follows that $cd(A_{\bar \Gamma})= dim(R(A_{\bar \Gamma}))$, and that this number depends only on the combinatorial structure of $\Gamma$ -- namely, $dim(R(A_{\bar \Gamma}))$ is the largest size of a clique in $\Gamma$.
\end{definition}

\begin{construction}\label{const:cubicalartin}
Let $R(A_{\bar \Gamma})=:X_\Gamma$, and consider the  cubical presentation
\begin{equation}\label{eq:artin1}
X_{\Gamma}^*=\langle X_\Gamma \mid Y_{ij} \rightarrow X_\Gamma \rangle
\end{equation}

 where,  for each $3\leq m_{ij} < \infty$, the complex $Y_{ij}$ is the Cayley graph of the 2-generator Artin group $A_{ij}=\langle v_i,v_j \mid (v_i,v_j)^{m_{ij}}=(v_j,v_i)^{m_{ij}} \rangle$  with respect to $\{v_i,v_j\}$ and $Y_{ij} \rightarrow X_\Gamma$ is the obvious covering projection on the $1$-skeleton. In the 2-dimensional case, the cubical presentations for Artin groups given in~\eqref{eq:artin1} were discussed by Wise in~\cite{WiseIsraelHierarchy}, and were anticipated in the work of Appel and Schupp~\cite{AppelSchupp83}.

 An intersection $gY_{ij}\cap g'Y_{i'j'}$ is either equal to all of $gY_{ij}$, in which case it does not contribute a piece, as $gY_{ij}$ and $g'Y_{i'j'}$ differ only by an automorphism, or $gY_{ij}\cap g'Y_{i'j'}$ is a bi-infinite geodesic $a_i^\infty$. Wall-pieces are intersections between the $gY_{ij}$'s and rectangles in $\widetilde X^*$. All flat strips come from tori in $X_\Gamma$, so if  no generator of $A_\Gamma$ commutes with  any pair $a_i$ and $a_j$ with $m_{ij}>4$, then a rectangle that intersects  $gY_{ij}$ intersects it on a bi-infinite geodesic, also of the form $a_i^\infty$. This amounts to requiring that no edge $(a_i, a_j)$ in $\Gamma$ forms a clique with any collection of edges in $\bar \Gamma$. 
Under these restrictions, even though cone-pieces and  wall-pieces are unbounded,  the systole  of $gY_{ij}$ is a loop spelling the dihedral Artin relation $(a_i,a_j)^{m_{ij}}((a_j,a_i)^{m_{ij}})^{-1}$, and so pieces arising in  an essential loop in $gY_{ij}$ have length $1$. So $X^*$ will satisfy the $C(9)$ condition provided that $m_{ij} \notin \{3,4\}$. However, without the restriction on the cliques of $\Gamma$, the wall-pieces in essential loops in $gY_{ij}$ can be large, so the argument fails.

To be able to extend the argument above, we utilise higher-dimensional relators. 
Let $\{K_1, \ldots, K_m\}$ be the set of maximal 2-joins in $\Gamma$ containing at least one edge in $\Gamma -\bar \Gamma$. Note that every edge $(a_i,a_j)$ with $m_{ij}\neq 2$ is contained in some $K_\ell$. For each $K_\ell$, let $\mathbf{Y}_{K_\ell}$ be the convex core of $Cay(A_{K_\ell})$ in $X_{\Gamma}$. % -- i.e., $mathbf{Y}_{K_1}$ is obtained from $Cay(A_{K_i})$ by gluing in a cube whenever its 1-skeleton is present
 Consider the following cubical presentation, which originated in joint forthcoming work with Alexandre Martin~\cite{ArenasMartin24}:
\begin{equation}\label{eq:artin2}
\mathbf{X}_{\Gamma}^*=\langle X_\Gamma \mid \mathbf{Y}_{K_\ell} \rightarrow X_\Gamma \rangle
\end{equation}
where $\mathbf{Y}_{K_\ell} \rightarrow X_\Gamma $ is induced by the map $\mathbf{Y}_{K_\ell}^{(1)} \rightarrow X_\Gamma $ given by the labels.

Reasoning as before, we see that while cone-pieces and  wall-pieces can be unbounded, the pieces arising in an essential loop in a $g\mathbf{Y}_{K_\ell}$ have length $1$. Again we deduce that the $C(9)$ condition holds provided that all $m_{ij} \notin \{3,4\}$. 
\end{construction}

We collect these observations as a lemma:

\begin{lemma}\label{lem:artin}
Suppose that $\Gamma$ satisfies the following: for all $i\neq j$, either $m_{ij}=2$ or $m_{ij}> 4$. Then $\mathbf{X}_\Gamma^*$ satisfies the $C(9)$ condition. 
\end{lemma}
 
Let $A_\Gamma$ be an Artin group satisfying that  no generator of $A_\Gamma$ commutes with any pair $a_i,a_j$ with $m_{ij} > 4$, and let $X^*$ be as in Construction~\ref{const:cubicalartin}. Then Theorem~\ref{thm:asph} implies that the reduced space $\bar X^*$ is aspherical. While the action of $A_\Gamma$ on $\bar X^*$ is not free -- indeed, each cone-point  stabiliser is an $A_{ij}$ --, the remark before Corollary~\ref{cor:asphericalreloaded} (or the corollary itself) can be applied to the cubical presentation  
\begin{equation}\label{eq:naiveartin}
W^*=\langle X_\Gamma \mid  \{c_{ij} \rightarrow X_\Gamma\} \rangle
\end{equation}
 where each $c_{ij}$ is a cycle spelling the word $(a_i,a_j)^{m_{ij}}((a_j,a_i)^{m_{ij}})^{-1}$ and $ c_{ij} \rightarrow X_\Gamma$ is the obvious immersion given by the labels. By construction, $A_\Gamma$ acts freely on  $\widetilde W^*$, so  $ W^*$  is a classifying space for  $A_\Gamma$. 
 For an arbitrary Artin group with no labels in $\{3,4\}$, we can also obtain a classifying space starting from the reduced space $\bar{\mathbf{X}}^*$ via Corollary~\ref{cor:asphericalreloaded} as follows.
Replace each cone  $Cone(\mathbf{Y}_{K_\ell})$ in  $\bar{\mathbf{X}}^*$ with a copy of a model for  $E(Stab(\mathbf{Y}_{K_\ell}))$ containing the corresponding $\mathbf{Y}_{K_\ell}$. 
For instance, if  $\mathbf{Y}_{K_\ell}=Cay(A_{ij})$ for some pair $(i,j)$, then $E(Stab(\mathbf{Y}_{K_\ell}))=\chi(A_{ij}, \{a_i,a_j\})$ is the presentation complex, which is aspherical by~\cite{Deligne72}, 
and also because $A_{ij}$ is a torsion-free one-relator group~\cite{Lyndon50,Cockcroft54}. 
In general, $E(Stab(\mathbf{Y}_{K_\ell}))$ is the universal cover of a classifying space for $A_{K_\ell}$, which is cocompact because $A_{K_\ell}$ decomposes as a direct product of 2-dimensional Artin groups and free groups.

 In~\cite{Charney00}, Charney proves the $K(\pi, 1)$ conjecture for all locally reducible Artin groups; this includes in particular the Artin groups in Lemma~\ref{lem:artin}. Our result partially recovers Charney's theorem, albeit via a purely topological viewpoint. Our results also allow us to compute most of the homology and cohomology groups of these Artin groups. Concretely, Corollary~\ref{cor:dimensions} and Proposition~\ref{prop:petsun1} give:

\begin{corollary}\label{cor:artins} Let $A_\Gamma$ be an Artin group with no labels in $\{3,4\}$. If no generator commutes with  any pair $a_i$ and $a_j$ with $m_{ij}\neq 2$, then:
$$H^n(A_\Gamma, M)=H^n(A_{\bar \Gamma},M) \text{ and } H_n(A_\Gamma, M)=H_n(A_{\bar \Gamma},M)$$ for every $A_\Gamma$-module $M$ and for all $n \geq 3$.  
Otherwise, for each maximal 2-join $K_\ell$ in $\Gamma$ containing at least one edge in $\Gamma -\bar \Gamma$, let $\bar K_\ell$ denote the subgraph of $K_\ell$ contained in $\bar\Gamma$ and let $N=\max\{|V(clique(\bar K_\ell))|\}$.  
Then: 
$$H^n(A_\Gamma, M)=H^n(A_{\bar \Gamma},M)\oplus \bigoplus_{\ell} H^n(A_{K_\ell},M)$$ and $$H_n(A_\Gamma, M)=H_n(A_{\bar \Gamma},M)\oplus \bigoplus_{\ell} H_n(A_{K_\ell},M)$$ for every $A_\Gamma$-module $M$ and for all $n \geq N+2$.  
\end{corollary} 

\begin{proof}
The first part of the corollary follows immediately from Corollary~\ref{cor:dimensions}. 
For the second part, note that $cd(Stab_{\pi_1X}(\widetilde{\mathbf{Y}_{K_\ell}} ))=\max\{|V(clique(\bar K_\ell))|\}$ where $K_\ell$ is as in Construction~\ref{const:cubicalartin} 
and that $Stab_{\pi_1X}(\widetilde{\mathbf{Y}_{K_\ell}} )/\pi_1\mathbf{Y}_{K_\ell}=A_{K_\ell}$. Apply Proposition~\ref{prop:petsun1}.
\end{proof}

\begin{remark}\label{rmk:artin acyl}
In the paper~\cite{ArHag2021}, Arzhantseva and Hagen give a sufficient condition for the fundamental group of a cubical small-cancellation  presentation to be acylindrically hyperbolic (Theorem~B therein). A list of examples satisfying this condition is also given. Artin groups $A_\Gamma$ with all $m_{ij}\in \{2\}\cup \naturals_{> 72}$, and  satisfying the condition that no edge $(a_i, a_j)$ in $\Gamma$ forms a clique with any collection of edges in $\bar \Gamma$ are included in this list via the cubical presentation $W^*$ in~\eqref{eq:naiveartin}. See Example~(3) in~\cite{ArHag2021}.
However, as explained in the previous paragraphs, this presentation does not even satisfy the $C(2)$ condition, so the acylindrical hyperbolicity cannot be deduced from this presentation. 
However,  Arzhantseva and Hagen's result does apply to the cubical presentation $X^*$ in~\eqref{eq:artin1} of these groups, and therefore, the acylindrical hyperbolicity can indeed be deduced: Theorem~A and Theorem~5.3 in~\cite{ArHag2021} hold without additional hypotheses, and even though $X^*$ does not directly satisfy all the hypotheses of  Theorem~B (indeed, $X^*$ is neither minimal nor has compact relators), it does satisfy a weak form of the \emph{uniform} $C''(\frac{1}{144})$ condition when  $m_{ij}\in \{2\}\cup \naturals_{> 72}$ for all $i<j$, and it appears\footnote{personal communication with Mark Hagen.} that in this setting, the remaining ingredient, a loxodromic element satisfying the conclusion of Lemma~6.7   in~\cite{ArHag2021} can be produced directly, using properties of RAAGs to prove Claim~9 without needing the extra hypotheses from Theorem~B.
\end{remark}

\subsection{Dyer and Shephard groups}
As in the previous subsection, let $\Gamma$ be a simplicial graph. We now label the edges \textbf{and} vertices of $\Gamma$ as follows: 
 edges $(a_i,a_j)$ are labelled with integers $m_{ij}\geq 2$ and  vertices $a_i$ are labelled with integers $m_i\geq 2$.

We impose an additional restriction on the labels:  if  $m_i \geq 3$ for some $i$, then  $m_{ij}=2$ for $i\neq j$. The \emph{Dyer group}  $D_\Gamma$ is given by the presentation: 
$$D_\Gamma=\langle a_1, \ldots, a_n \mid a_i^{m_i}=1: \text{ for all } a_i \in V(\Gamma) \text{ and }(a_i,a_j)^{m_{ij}}=(a_j,a_i)^{m_{ij}} : i\neq j \rangle.$$

When all $m_i=2$, then $D_\Gamma$ presents a \emph{Coxeter group}; when all $m_i=\infty$, then $D_\Gamma$ coincides with $A_\Gamma$.
In other words, Dyer groups simultaneously  generalise Coxeter groups and right angled Artin groups. When all $m_{ij}=2$ in the presentation given above, we say that $D_\Gamma$ is a \emph{right angled Dyer group}. A right-angled Dyer group admits a cubical presentation:
\begin{equation}\label{eq:dyer1}
X^*=\langle X_\Gamma \mid \{Y_{m_i} \rightarrow X_\Gamma\} \rangle
\end{equation}
where $Y_{m_i}$ is the cycle reading the word $a_i^{m_i}$ and the map $Y_{m_i}  \rightarrow X_\Gamma$ is the obvious immersion. This cubical presentation will satisfy the $C(9)$ condition if  $m_i\geq 9$ for every $a_i \in V(\Gamma)$ and $X_\Gamma$ is a graph, i.e., when $\Gamma$ has no edges -- as soon as $\Gamma$ has an edge, $X_\Gamma^*$ will have essential wall-pieces, and thus won't satisfy any cubical small-cancellation condition. To be able to cover more examples, we substitute the $Y_{m_i}$'s with non-compact relators.  

\begin{construction}\label{cons:dyer} Let $Z_{m_i}  \rightarrow X_\Gamma$ be the covering map corresponding to the subgroup  $\langle a_i^{m_i}\rangle$. Namely, $Z_{m_i}$ is a cylinder of width $m_i$. Now
consider  the cubical presentation:
\begin{equation}\label{eq:dyer2}
X^*=\langle X_\Gamma \mid \{Z_{m_i} \rightarrow X_\Gamma\} \rangle.
\end{equation}
$X^*$ satisfies the $C(9)$ condition provided that $m_i\geq 9$ for every $a_i \in V(\Gamma)$ and every edge has at most one vertex with label $<\infty$.
\end{construction}

The Davis-Moussong complex and the construction in~\cite{Soergel24} imply that Dyer groups are CAT(0), and in fact from~\cite[2.9]{Soergel24} one deduces that every Dyer groups is a finite index subgroup of a Coxeter group. In particular, the homology and cohomology  of Dyer groups can be understood in terms of that of Coxeter groups. Nevertheless, our approach provides an alternative viewpoint that utilises a rather different set of tools for understanding some of these groups. We can deduce for instance that

\begin{corollary}
Let $D_\Gamma$ be a Dyer group with $m_i\geq 9$ for every $a_i \in V(\Gamma)$ and where every edge has at most one vertex with label $<\infty$. Then
$$H^n(D_\Gamma, M)=H^n(A_{\bar \Gamma},M)\oplus \bigoplus_{v \in V_{\leq \infty}} H^n(\integers,M)$$ and $$H_n(D_\Gamma, M)=H_n(A_{\bar \Gamma},M)\oplus \bigoplus_{v \in V_{\leq \infty}} H_n(\integers,M)$$ for every $A_\Gamma$-module $M$ and for all $n \geq 4$.  
\end{corollary}

Dyer groups can be further generalised to \emph{Shephard groups} by dropping the defining restrictions on the labels.
For other (non-Dyer) Shephard groups, $C(9)$ cubical  presentations can sometimes also be produced by mixing Construction~\ref{const:cubicalartin} and Construction~\ref{cons:dyer}.

\subsubsection{Other quotients of RAAGs}
In general, let $X_\Gamma$ be the Salvetti complex of a RAAG $A_\Gamma$. For any $n \in \naturals$, one can produce a collection of ``small-cancellation words'' $g_1, \ldots, g_n \in A_\Gamma$ by picking a collection of independent, rank 1 isometries of  $\widetilde X_\Gamma$ and  raising these to large-enough powers, as in Example~\ref{const:covers}. Then, for each $i \in \{1, \ldots, n\}$ there exists a local isometry $Y_i \rightarrow X_\Gamma$ with $Y_i$ compact and  $\pi_1 Y_i=\langle g_i\rangle$. For each $i \in \{1, \ldots, n\}$, there exist finite-index covers $\hat Y_i \rightarrow Y_i$ such that the cubical presentation $X^*=\langle X_\Gamma \mid \{\hat Y_i\} \rangle$ satisfies the $C(9)$ condition. This construction is also explained in~\cite[Example (3)]{ArHag2021}.

\subsection{Examples in the hyperbolic setting} 

In the setting of ``classical'' group presentations, one way to produce small-cancellation is by raising the relators $r \in R$ of a presentation $\mathcal{P}=\langle S \mid R\rangle$ to large powers. In the cubical setting there is an analogous process:

\begin{example}\label{const:covers}(Forcing small-cancellation by taking covers)
If $X$ is a compact non-positively curved cube complex with hyperbolic fundamental group, and $H_1,...,H_k$ are quasiconvex subgroups of $\pi_1X$ that form a malnormal collection, then for each $n>0$ there are finite index subgroups $H'_i \subset H_i$ and local isometries $Y_i \rightarrow X$ with $Y_i$ compact and $\pi_1Y_i=H'_i$ such
that $X^*= \langle X \mid \{Y_i\} \rangle$ satisfies the cubical $C(n)$ condition. See~\cite[3.51]{WiseIsraelHierarchy}.  
\end{example}

When $n\geq9$ in Example~\ref{const:covers}, we conclude that $\bar X^*$ is a classifying space for proper actions for $\pi_1X^*$, and thus that $\pi_1X^*$ has rational cohomological dimension $cd_\rationals (\pi_1X^*) \leq cd(\pi_1 X)$. Being hyperbolic, $\pi_1 X$ is virtually special by Agol's Theorem (\cite[1.1]{AgolGrovesManning2012}), so Theorem~\ref{thm:embeds} actually implies that $cd_\rationals (\pi_1X^*)=vcd(\pi_1X^*) $ and that $\pi_1X^*$ is virtually torsion-free.

Again drawing on the analogy with classical small-cancellation theory, one may produce small-cancellation presentations by starting with a presentation $\mathcal{P}$ as before,  and multiplying each of the relators by long small-cancellation words that add ``noise'' to the presentation. This idea generalises as follows:

\begin{theorem}\label{const:noise}\cite[3.2]{Arenas2023}
Let $X$ be a compact special non-positively curved cube complex with $G =\pi_1X$ non-elementary hyperbolic. For all $k \geq 1$ there exist (infinitely many choices of) free non-abelian subgroups $\{H_1, \ldots, H_k\}$, and cyclic subgroups $\langle z_i \rangle <H_i$  such that the quotient $\Gamma=G/ \nclose{z_i}$ is hyperbolic and cocompactly cubulated. 
\end{theorem}

In a similar vein,  Futer and Wise have shown  that under suitable hypotheses, cubical small-cancellation quotients of cubulated hyperbolic groups are generically hyperbolic and cubulated:

\begin{theorem}\cite[1.1]{FuterWise16}\label{thm:futerwise} [Random cubical quotients]
Let $G = \pi_1X$, where $X$ is a compact non-positively curved cube complex, and $G$ is hyperbolic. 
Let $b$ be the growth exponent of $G$ with respect to the $\widetilde X$. Let $a$ be the maximal growth exponent of a stabiliser of an essential hyperplane of
$\widetilde X$.
Let $c< \min\{\frac{b-a}{20}, \frac{b}{41}\}$ and $k \leq e^{c\ell}$.
Then with overwhelming probability as $\ell \rightarrow \infty$, for any set of conjugacy classes $[z_1], \ldots , [z_k]$
with translation length $|z_i| \leq \ell$, the group $G/\nclose{z_1, \ldots , z_k}$ is hyperbolic and is the fundamental
group of a compact, non-positively curved cube complex.
\end{theorem}

While both Theorem~\ref{const:noise} and Theorem~\ref{thm:futerwise} produce a compact non-positively curved cube complex $C$ with $\pi_1 C=\Gamma$, and thus a finite classifying space, the dimension of $C$ is generically much larger than the cohomological dimension of $\Gamma$. In fact, even when $\Gamma$ is a classical small-cancellation group (and thus $vcd(\Gamma)\leq 2$), the dimension of \emph{any} non-positively curved cube complex with $\pi_1X=\Gamma$ can be arbitrarily large~\cite{Jank20}. Thus, a priori, we do not obtain any cohomological information (other than finite cohomological dimension) just from knowing that a group is cocompactly cubulated.

The proofs of Theorem~\ref{const:noise} and Theorem~\ref{thm:futerwise} use cubical small-cancellation to produce a  cubical presentation $X^*=\langle X \mid \{Y_i\} \rangle$ that satisfies in particular the $C(9)$ condition, and where the $Y_i$'s are compact, homotopy equivalent to cycles, and do not represent proper powers in $\pi_1 X$, so the coned-off space $X^*$  obtained from either construction is reduced. Thus, we deduce from Theorem~\ref{thm:asph} and Corollary~\ref{cor:dimensions}:

\begin{corollary}
Let $G = \pi_1X$, where $X$ is a compact non-positively curved cube complex, and $G$ is non-elementary hyperbolic. Let $\langle z_1 \rangle, \ldots, \langle z_n \rangle$ be a collection of cyclic subgroups of $G$ as in Theorem~\ref{const:noise} or  Theorem~\ref{thm:futerwise} and let $\Gamma=G/\nclose{z_1, \ldots, z_n}$. Then for each $i \in I$ there exist local isometries $Y_i \rightarrow X$ with $Y_i$ compact and  $\pi_1Y_i=\langle z_i \rangle$ such that the coned-off space of the cubical presentation  $X^*=\langle X \mid \{Y_i\}\rangle$ is a compact model for $K(\Gamma,1)$ of dimension 
$$dim(X^*)\leq \max\{\dim X, \max\{\dim(Y_i)+1\}\}.$$ 
Moreover, $ cd(G)-1 \leq cd(\Gamma)\leq \max\{cd(G),2\}$, $H^n(\Gamma, M)=H^n(G, M)$, and $H_n(\Gamma, M)=H_n(G, M)$ for every  $\Gamma$-module $M$ and for all $n \geq 3$.
\end{corollary} 

\section{Final comments}\label{sec:comments}

\subsection{Deducing asphericity from a CAT(0) structure}
Asphericity, and many other stronger properties, would follow  immediately if  $\bar X^*$ could be endowed with a CAT(0) metric. The following problem is thus of interest:

\begin{prob}\label{prob1}
For large enough $N$, does the reduced space associated to a  $C'(\frac{1}{N})$ cubical presentation admit a CAT(0) structure? If $\widetilde X$ is CAT(-1), does $\bar X^*$ admit a CAT(-1) structure? 
\end{prob}

Proving a result of this nature and in this generality is likely quite difficult. In the classical small-cancellation setting, results of this kind have been proven under the more restrictive \emph{uniform} $C''(\frac{1}{6})$ condition (see~\cite{brown2016cat-1,Gromov01, Martin17}). In the cubical setting, one may analogously define a \emph{uniform cubical} $C''(\frac{1}{N})$ condition by requiring that both the size of cone-pieces and wall-pieces are uniformly bounded. The following variation of Problem~\ref{prob1} might be more approachable.

\begin{prob}\label{prob2}
For large enough $N$, does the reduced space associated to a $C''(\frac{1}{N})$ cubical presentation admit a CAT(0) structure? If $\widetilde X$ is CAT(-1), does $\bar X^*$ admit a CAT(-1) structure? 
\end{prob}

\subsection{Below C(9)?}

In contrast with the classical version of the theory, where the $C(6)$ condition is already sufficient to prove a variety of results,  all the results in the   cubical setting assume at least the $C(9)$ condition. This is because Greendlinger's Lemma,  the main building block in the theory, holds under this hypothesis -- we shall now illustrate that the $C(9)$ hypothesis is in fact sharp, at least when the underlying cube complex $X$ for the cubical presentation $X^*$ has dimension greater than $2$; if $X$ is only assumed to have dimension at least $2$, a related example, described by Wise in~\cite[3.n]{WiseIsraelHierarchy}, exemplifies the failure of Greendlinger's Lemma in the cubical $C(6)$ case.

\begin{example} Let $X$ be the non-positively curved cube complex obtained from the 1-skeleton of the truncated cuboctahedron $P$ by gluing a square to every embedded 4-cycle in $P^{(1)}$, and gluing a 3-cube $c$ to every embedded 6-cycle in $P^{(1)}$ along an embedded  cycle in $c$ separating  $c^{(2)}$ into two ``halves''. Let $X^*$ be the cubical presentation obtained by coning-off each embedded 8-cycle in $P^{(1)}\subset X$, as in the  left of Figure~\ref{fig:notsmall} (note that only the ``front'' of $X$ and $X^*$ are drawn). Then $X^*$ satisfies the $C(8)$ condition.

Consider a disc diagram $D\rightarrow X^*$ obtained by choosing the 3 faces in the ``front'' or ``back'' of each 3-cube and removing a square that does not lie in a 3-cube,  as in the  right of Figure~\ref{fig:notsmall}. Then $D$ has minimal complexity, but does not satisfy any of the conclusion of Greendlinger's Lemma. Not only that, but it can be seen directly that the three-way decomposition of the cubical part of $X^*$ does not have the connectedness properties proven for the $C(9)$ case in Subsection~\ref{sec:inter}: indeed, there are supporting hyperplane carriers in $\hat X$ that have disconnected intersection with the cones. Finally, $X^*$ is clearly reduced, and yet not aspherical.
\end{example}

\begin{figure}[h!]
\centerline{\includegraphics[scale=0.5]{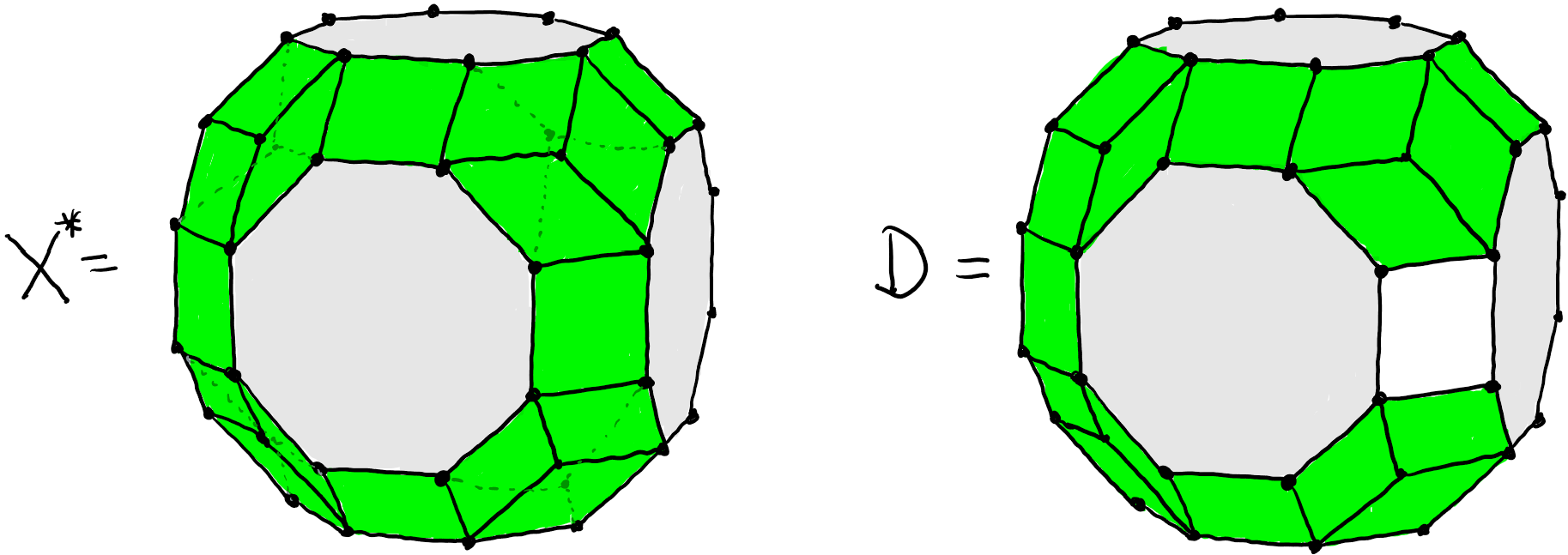}}
\caption{The cubical presentation $X^*$ is obtained by gluing squares and 3-cubes to the 1-skeleton of a truncated cuboctahedron. $X^*$ satisfies the $C(8)$ condition but not Greendlinger's Lemma, as can be seen by considering the minimal complexity disc diagram $D$ in the right of the figure.}
\label{fig:notsmall}
\end{figure}

\bibliographystyle{alpha}
%\bibliographystyle{plain}
%\bibliography{../wise}
%\bibliography{../../wise}

%\renewcommand{\refname}{}
%\bibliographystyle{amsalpha}
\bibliography{bib9.bib}
\end{document}